\documentclass[a4paper,11pt,reqno]{amsart}
\usepackage{amsmath,amsthm,amssymb}
\usepackage[dvips]{graphicx} 
\usepackage{psfrag}
\usepackage{color}
\usepackage[body={16cm,22.5cm},centering]{geometry} 
\usepackage{enumerate}
\usepackage{latexsym}
\usepackage{stmaryrd}
\usepackage[normalem]{ulem}

\theoremstyle{plain}
\newtheorem{theorem}{Theorem}[section]
\newtheorem{lemma}[theorem]{Lemma}
\newtheorem{proposition}[theorem]{Proposition}

\newtheorem*{conjecture*}{Conjecture}\theoremstyle{definition}
\newtheorem{definition}[theorem]{Definition}

\theoremstyle{remark}
\newtheorem{remark}[theorem]{Remark}

\numberwithin{equation}{section}

\newcommand{\R}{\mathbb{R}} 
\newcommand{\N}{\mathbb{N}}
\newcommand{\Z}{\mathbb{Z}}

\newcommand{\Om}{\Omega}

\newcommand{\dd}{\mathrm{d}}
\renewcommand{\div}{\operatorname{div}}
\newcommand{\eps}{\varepsilon}
\newcommand{\id}{\mathbf{id}}
\newcommand{\res}{\mathop{\hbox{\vrule height 7pt width .5pt depth 0pt \vrule height .5pt width 6pt depth 0pt}}\nolimits\,}

\renewcommand{\vec}[1]{\text{\boldmath $#1$}}
\newcommand{\vecg}[1]{\text{\boldmath $#1$}}

\newcommand{\A}{\mathcal{A}}
\newcommand{\As}{\mathcal{A}_s}
\newcommand{\Asr}{\mathcal{A}_s^r}
\newcommand{\oAsr}{\overline{\Asr}}

\newcommand{\p}{\partial}
\renewcommand{\t}{\theta}
\renewcommand{\d}{\delta}

\newcommand{\mc}{\mathcal}

\newcommand{\weakc}{\rightharpoonup}
\newcommand{\weakcs}{\overset{*}{\rightharpoonup}}

\DeclareMathOperator{\cof}{cof}
\DeclareMathOperator{\Det}{Det}
\DeclareMathOperator{\imG}{im_G}
\DeclareMathOperator{\imT}{im_T}
\DeclareMathOperator{\imTBN}{im_T^{BN}}
\DeclareMathOperator{\loc}{loc}
\DeclareMathOperator{\adj}{adj}
\DeclareMathOperator{\supp}{supp}
\DeclareMathOperator{\dist}{dist}
\DeclareMathOperator{\loca}{loc}

\DeclareMathOperator{\rel}{rel}
\DeclareMathOperator{\tr}{tr}
\DeclareMathOperator{\sgn}{sgn}
\def\R{\mathbb{R}}

\def\E{\mathcal E}

\def \p{\partial}

\def \t {\theta}
\def \e {\varepsilon}

\def \O {\Omega}

\mathchardef\emptyset="001F

\begin{document}
\title[Relaxation of the neo-Hookean energy in 3D]{Harmonic dipoles and the relaxation of the neo-Hookean energy in 3D elasticity}
\author{Marco Barchiesi}
\author{Duvan Henao}
\author{Carlos Mora-Corral}
\author{R\'emy Rodiac}
\date{\today}

\address[Marco Barchiesi]{
Dipartimento di Matematica e Geoscienze, Universit\`a degli Studi di Trieste,
Via Weiss 2 - 34128 Trieste, Italy.
}
\email{barchies@gmail.com}

\address[Duvan Henao]{Faculty of Mathematics \(\&\) Institute for Mathematical and Computational Engineering, Pontificia Universidad Católica de Chile,
Vicu\~na Mackenna 4860, Macul, Santiago, Chile. (Current address: Instituto de Ciencias de la Ingenier\'ia, Universidad de O'Higgins, Rancagua, Chile.)}
\email{dahenao@uc.cl}

\address[Carlos Mora-Corral]{Departamento de Matem\'aticas, Universidad Aut\'onoma de Madrid,
28049 Madrid, Spain  and Instituto de Ciencias Matem\'aticas,
CSIC-UAM-UC3M-UCM, 28049 Madrid, Spain.
}
\email{carlos.mora@uam.es}

\address[R\'emy Rodiac]{
Universit\'e Paris-Saclay, CNRS\\
Laboratoire de Math\'ematiques d'Orsay\\
91405, Orsay, France\\
}
\email{remy.rodiac@math.universite-paris-saclay.fr}

\begin{abstract}
We consider the problem of minimizing the neo-Hookean energy in \(3D\). The difficulty of this problem is that
the space of maps without cavitation is not compact, as shown by Conti \& De Lellis with a pathological example involving a dipole.
In order to rule out this behaviour we consider the relaxation of the neo-Hookean energy in the space of axisymmetric maps without 
cavitation. We propose a minimization space and a new explicit energy penalizing the creation of dipoles. 
This new energy, which is a lower bound of the relaxation of the original energy, bears strong similarities with the relaxed energy 
of Bethuel-Brezis-H\'elein in the context of harmonic maps into the sphere.
\end{abstract}

\maketitle


\section{Introduction}

\subsection{A regularity problem for the well-posedness of the neo-Hookean model}

We consider the problem of the existence of minimizers for the neo-Hookean energy, i.e.,
\begin{equation}\label{eq:E}
E(\vec u)=\int_\Om \left[ |D \vec u|^2+ H(\det D \vec u) \right] \dd\vec x 
\end{equation}
where \(H: (0, \infty) \rightarrow [0, \infty) \) is a convex function such that 
\begin{align}
	\label{eq:explosiveH}
\lim_{t\rightarrow \infty} \frac{H(t)}{t}=\lim_{s \rightarrow 0} H(s)=\infty,
\end{align}
\( \Om \subset \R^3 \) represents the reference configuration of an elastic body, and  \( \vec u: \Om \rightarrow \R^3 \) 
is the deformation map.
The neo-Hookean energy is widely used in physics, engineering and materials science and it can be derived from first principles \cite{Treloar75,Weiner17}
assuming that the gradient of the deformation remains bounded. Since interpenetration of matter is physically unrealistic,
minimizers are sought in a suitable subclass of 
\begin{equation*}
 \A := \{ \vec u \in H^1(\Om,\R^3) : \, E(\vec u)<\infty, \, \vec u \text{ is injective a.e.} \}.
\end{equation*}

In his celebrated existence theory, Ball \cite{Ball77} was able to apply the direct method of the calculus of variations to 
general polyconvex energies 
\begin{equation}\label{eq:generalW}
\int_\Om W( \vec x, D \vec u( \vec x )) \, \dd \vec x,
\end{equation}
where \( W:\Om \times \R^{3 \times 3} \rightarrow \R \cup\{\infty \} \)
is the elastic stored-energy function of the material. His approach is based on the identity 
\begin{equation*}
	\det D\vec u = \Det D\vec u, \quad \langle \Det D\vec u, \varphi \rangle := - \frac{1}{3} \int_\Omega  \vec u(\vec x) \cdot \big ( (\cof D\vec u) D\varphi \big ) \, \dd \vec x ,\quad \varphi \in C_c^1(\Omega),
\end{equation*}
whereby the Jacobian determinant can be written as a distributional divergence, and an analogous identity for $\cof D\vec u$, 
the matrix of $2\times 2$ cofactors of the deformation gradient. 
The identities are obtained from a coercivity assumption on the stored-energy function, the sharpest version of it (due to M\"uller, Tang \& Yan \cite{MuQiYa94}) being that 
\begin{equation*}
 W(\vec x, \vec F) \geq c_1 |\vec F|^2+c_2|\cof \vec F|^{3/2} + H(\det \vec F) , \qquad (\vec x, \vec F) \in \O \times \R^{3 \times 3}.
\end{equation*}
However, this coercivity excludes the neo-Hookean materials.
In fact, for neo-Hookean materials the hypothesis of finite energy alone is insufficient to ensure that $\det D\vec u = \Det D\vec u$, 
as shown in the models for cavitation \cite{Ball1982,Sverak88,MuSp95,SiSp00}. 
Because of that, the neo-Hookean energy is not $H^1$-quasiconvex, which is necessary for \eqref{eq:generalW} 
to be $H^1$-weakly lower semicontinuous in $\A$, as proved by Ball \& Murat \cite{BallMurat1984}. 

In order to overcome the lack of $H^1$-weakly lower semicontinuity of the neo-Hookean energy in  $\A$ due to cavitation, one may
look for minimizers in the smaller class $\A^r$ of maps in $\A$ for which the divergence identities 
\begin{align}
\label{eq:divergence_identities}
\mathrm{Div}\, \big ( (\adj D\vec u)\vec g\circ \vec u \big ) = (\div \vec g)\circ \vec u \,\det D\vec u\quad \forall\,\vec g \in C_c^1(\R^3,\R^3)
\end{align}
(of which $\det D\vec u=\Det D\vec u$ is a particular case) are satisfied, see e.g.\ \cite{GiMoSo89}, \cite{Muller90}, \cite{HeMo10}.
Unluckily, one has then to face a problem of lack of compactness:
Conti \& De Lellis \cite{CoDeLe03} constructed a sequence of deformations satisfying the divergence identities \eqref{eq:divergence_identities}, 
weakly converging in \(H^1\) but such that the limit does not satisfy \eqref{eq:divergence_identities}. 
In this paper we try to overcome this obstruction.

Since we want to rule out the formation of anomalies at the boundary we assume
that $\Om \Subset \widetilde{\Om}$, where $\widetilde{\Om}$ is a smooth bounded domain of $\R^3$,
and require the deformations $\vec u$ to coincide with a bounded $C^1$ orientation-preserving
diffeomorphism $\vec b:\widetilde{\O} \rightarrow \R^3$ not only on $\partial \Omega$ but on the whole of
$\widetilde\Omega \setminus \Omega$, and to be injective a.e.\ on the whole of $\widetilde\Omega$.
This setting was used before in elasticity \cite{Sverak88,SiSp00,HeMoXu15}, and can nevertheless be avoided 
with the techniques of \cite{HeMoOl21}. 

Set
\[ \O_{\vec b} := \vec b (\O) , \qquad \widetilde{\O}_{\vec b} := \vec b (\widetilde{\O}) .
\]
Since the example of Conti \& De Lellis is axisymmetric, 
we assume that $\O$, $\widetilde{\O}$ and $\vec b$ are axisymmetric (see the definition \eqref{eq:axi} in Section \ref{subse:axisymmetric}).
Define 
\begin{multline}
	\label{eq:defAs_introduction}
\As :=\{ \vec u \in H^1(\widetilde{\Om},\R^3): \, \vec u \text{ is injective a.e.~and axisymmetric,}
\\
 \det D\vec u>0 \ \text{a.e.},\ 
\vec u = \vec b \text{ in } \widetilde{\Om} \setminus \Om, 
\ \text{and}\  E(\vec u)\leq E(\vec b)  \}.
\end{multline}

Interestingly, in the axisymmetric setting we can prove that \(E\) is weakly lower semicontinuous and then that
it has a minimum in \(\As \) (see Proposition \ref{prop:closedeness_of_Asym}).
However, the weak limit in the Conti-De Lellis example belongs to \(\As\). That example exhibits a \emph{dipole singularity}, i.e.,
a cavity opened at a point is filled by material coming from a small neighbourhood of another point. 
Such a flagrant interpenetration of matter can hardly be accepted as physical. 
Because of that, in building an existence theory for the neo-Hookean energy we would like to prove more regularity on 
minimizers by showing their existence in the class 
\begin{align}
	\label{eq:defAsr}
\Asr := \{\vec u \in \As: \text{the divergence identities \eqref{eq:divergence_identities} are satisfied} \} .
\end{align}

In order to minimize \(E\) in \(\Asr\) we employ a relaxation process. 
The reader can think of the minimization of a functional in $W^{1,1}$. Since that space is not weakly compact, one sets up the
problem in the larger space $BV$ of functions of bounded variation, and relaxes the functional by adding a term that takes into 
account the singular part of the distributional gradient. Here we propose something similar, but the singular part appears on the 
inverse of the deformation.
More precisely, we set up the problem in the space
\begin{align}\label{eq:defB}
\mathcal{B}:=\{\vec u \in \As: \ \widetilde{\Omega}_{\vec b} = \imG(\vec u,\widetilde{\Omega}) \text{ a.e.\ and } 
\vec u^{-1}=(u^{-1}_1,u^{-1}_2,u^{-1}_3)\in W^{1,1}(\widetilde{\Om}_{\vec b}, \R^2) \times BV(\widetilde{\Om}_{\vec b}) \} ,
\end{align}
where the geometric image $\imG(\vec u, \widetilde \Omega)$ is that of Definition \ref{def:geometric_image}.
The space $\mathcal{B}$ contains the weak $H^1$ closure of $\Asr$ (see Theorem \ref{prop:image_ouverte}).
The fact that in $\mathcal{B}$ the geometric image coincides with $\widetilde{\Omega}_{\vec b}$ means that any cavitation 
produced by a map in this class must be filled with the image of some other part of the body (as happens in the example of 
Conti \& De Lellis).

In $\mathcal{B}$ we provide a lower bound for the relaxation of $E$:
$$
F(\vec u):=\int_\Omega \left[|D\vec u|^2 + H(\det D\vec u)\right] \dd \vec x 
+ 2 \left| D^s u^{-1}_3 \right| (\tilde{\O}_{\vec b}) ,
$$
where \(|D^s u^{-1}_3|(\widetilde{\O}_{\vec b})\) denotes the total variation of the singular part of the 
distributional gradient of $u^{-1}_3$.
We believe that this lower bound is sharp, i.e.,  that $F$ coincides with the relaxation of $E$.
Indeed, in the companion paper \cite{Barchiesi_Henao_MoraCorral_Rodiac_b} we improve the construction 
of Conti \& De Lellis, by showing that the relaxation of $E$ on that limit map (with a dipole) coincides with $F$. 
Proving the sharpness is important in order to get, eventually, a negative result: if the minimizers of 
the relaxed energy do not belong to $\Asr$, then $E$ has no minimizers in $\Asr$.
In any case, the energy and the space we propose can serve to the purpose of providing a positive result,
i.e., existence of minimizers for $E$.

\begin{theorem} \label{th:main_theorem_introduction}
The energy $F$ has a minimizer in $\mathcal{B}$. Moreover, if it belongs to $\Asr$,
then it is also a minimizer of the original neo-Hookean energy $E$ of \eqref{eq:E}.
\end{theorem}

The new term $2|D^s u^{-1}_3|(\tilde{\O}_{\vec b})$ is the main contribution of this work. 
The strategy we propose to answer the question of existence of minimizers of the neo-Hookean energy in a class 
of regular maps is to obtain Sobolev regularity for the inverse \(\vec u^{-1}\) of a minimizer \(\vec u\) of \(F\) 
in \(\mathcal{B}\). This is left for future investigations. We remark that the advantage of our results is that 
both $\mathcal{B}$ and $F$ are explicit.

\subsection{The singular energy}

We explain here how the singular term \( |D^s u^{-1}_3|(\tilde{\O}_{\vec b})\) appears in Theorem \ref{th:main_theorem_introduction} 
and why we believe it is the adequate term to add to minimize the neo-Hookean energy, at least in the axisymmetric setting. 
Indeed, the only way in which the divergence identities do not pass to the limit is when the cofactors are not equiintegrable.
In this, the behaviour of the sequence in the example of Conti \& De Lellis is generic: a `stack' of surfaces with smaller and 
smaller diameters, orthogonal to the axis of symmetry,  are stretched without control.
Due to the \(2D\) nature of the axisymmetric setting (see Lemma \ref{lem:inequality_energy_area_true}), we have
$$
	\int_{C_\delta} |D\vec u_j|^2 \, \dd\vec x \geq 
	2
	\int_{C_\delta} \left| (\cof D\vec u_j)\vec e_3 \right| \dd \vec x 
	= 2\int_{\vec u_j(C_\delta)} \left| D(u^{-1}_3)_j \right| \dd\vec y
$$
where $\vec e_3$ is the direction of the symmetry axis and $C_{\d}$ is a small $\delta$-cylinder around it. 
The sets $\vec u_j(C_\delta)$ collapse to a set with  zero volume (thanks to the equiintegrability of the determinants); 
in the example by Conti \& De Lellis, they collapse to a sphere, 
which is exactly the jump set of the vertical component of the inverse $u_3^{-1}$ for the limit map $\vec u$.
However,
$$
	\int_\Omega |D\vec u|^2 \, \dd\vec x 
	=  \lim_{\delta \searrow 0} \int_{\Omega\setminus C_\delta} |D\vec u|^2 \, \dd\vec x 
	\leq  \liminf_{\delta \searrow 0}
	\liminf_{j\to\infty} \int_{\Omega\setminus C_\delta} |D\vec u_j|^2 \, \dd\vec x , 
$$
so the original formula for the neo-Hookean energy completely misses out
the concentration of the Dirichlet energy if applied directly to the singular map $\vec u$. 

Note that in the reference configuration all evidence of the abnormal activity of the regular sequence is lost and 
hidden in the one-dimensional symmetry axis; in contrast, in the deformed configuration large (two-dimensional or fractal) 
structures can remain, which make the singular map remember the energy spent in their formation. All in all, the singular term 
$2|D^s u^{-1}_3|(\widetilde{\O}_{\vec b})$
is \emph{not artificial}, it emerges naturally from the Dirichlet energy.
At the very least, we prove, cf.\ equation \eqref{eq:lowerbound}, that it is a lower bound of the abstract relaxed energy functional.

We want to make a last comment about recovering the regularity in order to obtain minimizers of the neo-Hookean energy
in the original space $\Asr$. In this task we will be confronted not with singularities that are physically relevant but 
with pathological deformations that we would prefer to exclude. 
Indeed, in the companion paper \cite{Barchiesi_Henao_MoraCorral_Rodiac_b}, under the additional assumption 
$\mathcal E(\vec u)<\infty$ for the surface energy functional defined in Section \ref{se:surface_energy}, we prove 
that for any weak limit of regular maps there exist a countable family of dipoles \(\vec \xi_i, \vec \xi_i'\) lying on the 
axis of symmetry and a countable family of sets of finite perimeter whose reduced boundaries \(\Gamma_i\) satisfy 
\begin{equation}\label{multi dipoles}
|D^s u^{-1}_3|(\widetilde{\O}_{\vec b})=\sum_{i\in \mathbb{N}} \left| \vec \xi_i - \vec \xi_i' \right| \mathcal{H}^2(\Gamma_i).
\end{equation}
The discussion is therefore about how to reach a contradiction from the assumption that a minimizing sequence of regular 
maps ends up forming those dipoles, and a successful argument could probably use that if that were the case then the 
regular maps in the sequence would produce an energy concentration of (at least)
\begin{equation}\label{eq:bubble}
	2|D^su^{-1}_3| (\widetilde{\O}_{\vec b}) = 2\cdot (\text{area of the bubble})
	\cdot (\text{length of the dipole}),
\end{equation}
which is presumably more than what a minimizer can afford.

\subsection{Connection with harmonic map theory}

Bethuel, Brezis and Coron \cite{BeBrCo90} (see also \cite{GiMoSo98II}) also derived a relaxed energy to treat a problem 
of lack of compactness in the theory of harmonic maps from a \(3D\) domain with values into \(\mathbb{S}^2\). 
The expression \eqref{eq:bubble} shows that the energy we obtain and the relaxed energy in the context of harmonic maps are very similar. In particular, the right-hand side of \eqref{multi dipoles} is the analogue of the `length of minimal connection', 
introduced in \cite{BrCoLi86}, connecting singularities of harmonic maps.
Besides, the supplementary term in the harmonic map relaxed energy can be expressed in terms of this length of minimal connection in the case where the map has a finite number of singularities. This reveals a strong connection between the problem of minimizing the neo-Hookean energy and finding a smooth minimizing harmonic map from \(\mathbb{B}^3\) into \(\mathbb{S}^2\) with a smooth boundary data with zero degree. This problem was raised by Hardt and Lin in \cite{Hardt_Lin_1986} and is still open. For the study of partial regularity and prescribed singularities problem for harmonic maps from \(\mathbb{B}^3\) to \(\mathbb{S}^2\) in the axisymmetric setting we refer to \cite{Hardt_Lin_Poon_1992} and \cite{Martinazzi_2011}.

\subsection{Recent related results}
During the process of revision of this paper, we became aware of the recent works \cite{DoHeMa21,DoHeMo22}.
Both exhibit coercivity conditions on $W$ so that condition INV (see Definition \ref{def:INV}) is preserved under the weak limit in $W^{1,N-1}$: 
through a quick enough growth to infinity when the determinant goes to zero in \cite{DoHeMa21}, and through the equiintegrability of the cofactors in \cite{DoHeMo22}.
In fact, that the equiintegrability of the cofactors implies the stability of condition INV was shown in \cite{HeMo12}.
In addition, in \cite{DoHeMa21} they construct an example of a map in the Conti--De Lellis style.

\subsection{Outline of the paper} The paper is organized as follows. In Section \ref{sec:prel} we introduce some notation and definitions that will be used in the sequel. 
More precisely, we define the geometric image and the surface energy of a map.
The latter notion quantifies the failure of the divergence identities \eqref{eq:divergence_identities}.
We then make precise the axisymmetry and specify the boundary condition.
We also introduce the notion of family of `good open sets' and show how to relate the properties of 
a \(3D\) axisymmetric map to the properties of its associated \(2D\) map. 

Section \ref{sec:existence_as} is devoted to the proof of existence of minimizers of $E$ in the class \(\As\). These minimizers could, in principle, be irregular, 
forming pathologies similar to that of the example of Conti--De Lellis. This leads to the question of whether conditions exist under which such behaviour 
can be ruled out. With this motivation in mind, our main focus in this paper is to derive an explicit energy playing the role of a relaxed energy for the 
neo-Hookean problem, in which the cost of creating pathological singularities can be made visible.

In order to do that, in Section \ref{se:fine_properties} we describe fine properties of maps in \( \As\).
We start with regularity properties of general axisymmetric maps, then we define the topological degree and topological image of maps. We will need both definitions of the classical degree for continuous functions and of the Brezis--Nirenberg degree for Sobolev maps.
A particular role is played by the topological image of the segment formed by the intersection of the domain \(\O\) and the symmetry axis.
Then we focus on the invertibility property of maps in \(\As\).
The main result of that section states that the first two components of the inverse of a map in \(\As\) are Sobolev.

In Section \ref{sec:weak_limits}, we focus first on regularity properties of weak limits of maps in \(\Asr\).
It is of crucial importance for the rest of the paper that their geometric image equals (up to a null \(\mathcal{L}^3\)-set) the entire target domain, and that their inverses

are in \(BV(\widetilde{\O}_{\vec b},\R^3)\), with the first two components in \(W^{1,1}(\widetilde{\O}_{\vec b})\).
These results rest on the preliminary analysis done in Section \ref{se:fine_properties}.

Finally, in Section \ref{sec:lower_bound}, we give a lower semicontinuity result for our candidate relaxed energy, 
hence proving a lower bound on the actual relaxed energy. We also obtain various existence results thanks to the 
previous analysis and give a proof of Theorem \ref{th:main_theorem_introduction}.

\section{Notation and preliminaries}\label{sec:prel}

\subsection{Geometric image and area formula}

In this section \(\Om\) is a bounded open set of \(\R^N\). 
We use the following notation for the density of a measurable set \(A\subset \R^N\) at \(\vec x \in \R^N\):
\[
 D(A,\vec x)=\lim_{r\to 0} \frac{|B(\vec x,r)\cap A|}{|B(\vec x,r)|} .
\]
Here we use $\left| \cdot \right|$ for the Lebesgue measure in $\R^N$.
An alternative notation is $\mathcal{L}^N$.
The Hausdorff measure of dimension $d$ is denoted by $\mathcal{H}^d$.
The abbreviation \emph{a.e.}\ for \emph{almost everywhere} or \emph{almost every} will be intensively used.
It refers to the Lebegue measure, unless otherwise stated.
Given two sets $A, B$ of $\R^N$, we write $A \subset B$ a.e.\ if $\mc{L}^N (A \setminus B) = 0$, while $A = B$ a.e.\ or $A \overset{\text{a.e.}}{=} B$ a.e.\ both mean 
$A \subset B$ a.e.\ and $B \subset A$ a.e\@.
An analogous meaning is given to the expression $\mathcal{H}^d$-a.e\@. 

The definition of approximate differentiability can be found in many places (see, e.g., \cite[Sect.\ 3.1.2.]{Federer69}, \cite[Def.\ 2.3]{MuSp95} or \cite[Sect.\ 2.3]{HeMo12}).

We recall the area formula of Federer (\cite[Prop.\ 2.6]{MuSp95} and \cite[Thm.\ 3.2.5 and Thm.\ 3.2.3]{Federer69}). We will use the notation $\mc{N}(\vec u, A,\vec y)$ for the number of preimages of a point $\vec y$ in the set $A$ under~$\vec u$.

\begin{proposition}\label{prop:area-formula}
Let $\vec u\in W^{1,1}(\O,\R^N)$, and denote the set of approximate differentiability points of $\vec u$ by $\O_d$. Then, for any measurable set $A\subset \O$ and any measurable function $\varphi:\R^N \rightarrow \R$,
\begin{equation*}
\int_A (\varphi \circ \vec u) \left| \det D \vec u \right| \dd \vec x =\int_{\R^N} \varphi(\vec y) \, \mc{N}(\vec u,\O_d\cap A,\vec y) \, \dd \vec y
\end{equation*}
whenever either integral exists. Moreover, if a map $\psi:A\rightarrow \R$ is measurable and $\bar{\psi}:\vec u(\O_d\cap A) \rightarrow \R$ is given by
\begin{equation*}
\bar{\psi}(\vec y):= \sum_{\vec x \in \O_d \cap A, \ \vec u(\vec x)= \vec y} \psi(\vec x)
\end{equation*}
then $\bar{\psi}$ is measurable and
\begin{equation}\label{eq:area-formula}
\int_A\psi(\varphi\circ \vec u) \left| \det D \vec u \right| \dd \vec x= \int_{\vec u(\O_d\cap A)} \bar{\psi} \varphi \, \dd \vec y, \ \ \vec y \in \vec u(\O_d\cap A),
\end{equation}
whenever the integral on the left-hand side of \eqref{eq:area-formula} exists.
\end{proposition}

\begin{definition}\label{def:Om0}
Let $\vec u\in W^{1,1}(\O,\R^N)$ be such that $\det D \vec u>0$ a.e. We define $\O_0$ as the set of $\vec x\in \O$ for which the following are satisfied:
\begin{enumerate}[i)]
\item the approximate differential of $\vec u$ at $\vec x$ exists and equals $D \vec u(\vec x)$.
\item there exist $\vec w\in C^1(\R^N,\R^N)$ and a compact set $K \subset \O$ of density $1$ at $\vec x$  such that $\vec u|_{K}=\vec w|_{K}$ and $D \vec u|_{K}=D \vec w|_{K}$,
\item $\det D \vec u(\vec x)>0$.
\end{enumerate}
We note that the set \( \O_0\) is a set of full Lebesgue measure in \( \O \), i.e., \( |\O \setminus \O_0|=0\). This follows from Theorem 3.1.8 in \cite{Federer69}, 
Rademacher's Theorem and Whitney's Theorem.

Two important properties for a map are Lusin's properties (N) and (N$^{-1}$).

\begin{definition}
Let \(X \subset \R^N\) be a measurable set. We say that a measurable function \( \vec u: X \rightarrow \R^N \) satisfies Lusin's condition (N) 
if for every \(A \subset X\) such that \(|A|=0\) we have \(|\vec u(A)|=0\).
We say that \( \vec u\) satisfies condition (\(\text{N}^{-1}\)) if for every \(A\subset \R^N\) such that \(|A|=0\) we have \(|\vec u^{-1}(A)|=0\).
\end{definition}

We will use the following consequence of Proposition \ref{prop:area-formula} (see, e.g., \cite[Lemma 2.8]{BaHeMo17}).

\begin{lemma}\label{le:Lusin}
Let $\vec u \in W^{1,1} (\O, \R^N)$.
Then $\vec u|_{\O_0}$ satisfies Lusin's condition (N).
Moreover, if  $\det D\vec u (\vec x) \neq 0$ for a.e.\ $\vec x \in \O$, then $\vec u$ satisfies Lusin's (N$^{-1}$) condition.
\end{lemma}

\end{definition}
\begin{definition}\label{def:geometric_image}
For any measurable set $A$ of $\O$, the geometric image of $A$ under $\vec u$ is  
\begin{equation*}
\imG(\vec u,A) : =\vec u(A\cap \O_0) ,
\end{equation*}
with \(\Om_0\) as in Definition \ref{def:Om0}.
\end{definition}

\subsection{The surface energy}
	\label{se:surface_energy}
In this subsection \(N \in \mathbb{N}\) and \(\O \subset \R^N\) is a bounded open set. Let $ \vec u\in W^{1,N-1}(\O,\R^N)$ be such 
that $\det D \vec u \in L^1(\O)$.

The adjugate matrix $\adj \vec F$ of $\vec F \in \R^{N \times N}$ satisfies $(\det \vec F) \vec{I} = \vec F \adj \vec F$, where $\vec{I}$ denotes the identity matrix.
The transpose of $\adj \vec F$ is the cofactor $\cof \vec F$.
We start by observing that, when $N=3$, \( \left| \cof \vec F \right|\) is controlled in terms of \( |\vec F|^2\).
The proof of the following result is elementary and based on singular value decomposition.
\begin{lemma}\label{eq:ineq-energy-area}
$|\vec F|^2 \geq \sqrt{3} \left| \cof \vec F \right|$ for all $\vec F \in \R^{3 \times 3}$, with optimal constant.
\end{lemma}

\begin{definition}
Let $ \vec u\in W^{1,1} (\O,\R^N)$ be such that $\cof D \vec u \in L^1(\O, \R^{N \times N})$ and $\det D \vec u \in L^1(\O)$.

\begin{enumerate}[a)]
\item For every $\phi \in C^1_c(\O)$ and $\vec g\in C^1_c(\R^N,\R^N)$ we define
\begin{equation*}
\overline{\E}_{\vec u }(\phi,\vec g)=\int_\O\left[ \vec g( \vec u(\vec x))\cdot \left( \cof D \vec u(\vec x) D\phi(\vec x) \right) + \phi( \vec x)\div \vec g ( \vec u(\vec x)) \det D \vec u(\vec x) \right] \dd \vec x .
\end{equation*}

\item For all $\vec f \in C^1_c(\O\times \R^N,\R^N)$ we define 
\begin{equation*}
\E_{\vec u}(\vec f)= \int_\O \left[ D_\vec x\vec f(\vec x,\vec u(\vec x))\cdot \cof D\vec u(\vec x)+\div_\vec y \vec f(\vec x,\vec u(\vec x)) \det D \vec u(\vec x) \right] \dd \vec x 
\end{equation*}
and
\[
 \E(\vec u)= \sup \{ \E_{\vec u}(  \vec f): \, \vec f \in C^1_c(\O\times \R^N,\R^N) , \, \| \vec f \|_{L^{\infty}} \leq 1 \} .
\]
\end{enumerate}
\end{definition}

Sometimes we will use the notation $\E (\vec u, V)$ to refer to $\E (\vec u |_V)$, where $V$ is an open subset of $\O$.
Clearly, $\overline \E \leq \E$.
The following result shows that if $\overline \E$ vanishes, so does $\E$. 
This implies that 
the divergence identities \eqref{eq:divergence_identities} are satisfied if and only if the surface energy $\mathcal E$ is identically zero.
The proof consists in using the continuity of \( \vec f \mapsto \E_{ \vec u} ( \vec f )\) and the density of the linear span of products of 
functions of separated variables in \( C^1_c(\R^N,\R^N)\) (see, e.g., \cite[Cor.\ 1.6.5]{Llavona86}).

\begin{lemma}\label{lem:link_Ebar_E}
Let $ \vec u\in W^{1,N-1}(\O,\R^N)$ be such that $\det D \vec u \in L^1(\O)$.
If $\overline{\E}(\vec u)=0$ then $\E(\vec u)=0$.
\end{lemma}

In the rest of this section we take \(N=3\). The following result is a particular case of the calculation of the energy $\E$ for the Cartesian product of two functions.

\begin{lemma}\label{le:Evid}
Let $\omega \subset \R^2$ and $I \subset \R$ be both open and bounded.
Let $\vec v \in H^1 (\omega, \R^2)$.
Then $\det D \vec v \in L^1 (\omega)$ and $\E (\vec v) = 0$.
Define $\vec w : \omega \times I \to \R^3$ as $\vec w (x_1, x_2, x_3) = (\vec v (x_1, x_2), x_3)$.
Then $\E (\vec w) = 0$.
\end{lemma}

\begin{proof}
Clearly, \( \left| \det D \vec v \right |\leq \frac12 |D \vec v|^2 \in L^1(\omega)\).
The fact that \(\mathcal{E}(\vec v)=0\) is standard and can be shown by approximation by smooth maps and integration by parts (e.g., \cite[Lemma 2]{Muller88} or \cite[Th.\ 4.2]{MuQiYa94}).
For $\vec x \in \R^3$, write 
$\vec x = (\hat{\vec x}, x_3)$ with $\hat{\vec x} = (x_1, x_2)$, and analogously for $\vec y \in \R^3$.
 Using Lemma \ref{lem:link_Ebar_E}, it suffices to show that $\bar{\E}_{\vec w} (\phi, \vec g) = 0$ for $\phi \in C^1_c (\omega \times I)$ 
 and $\vec g \in C^1_c (\R^3, \R^3)$ of the form
\[
 \phi(\vec x) = \phi_1 (\hat{\vec x}) \, \phi_3 (x_3) \quad \text{and} \quad \vec g (\vec y) = \binom{\vec g_{11} (\hat{\vec y}) \, g_{13} (y_3)}{g_{31} (\hat{\vec y}) \, g_{33} (y_3)},
\]
for some
\[
 \phi_1 \in C^1_c (\omega), \ \phi_3 \in C^1_c (I), \ \vec g_{11} \in C^1_c (\R^2, \R^2), \ g_{13} \in C^1_c (\R) , \ g_{31} \in C^1_c (\R^2) , \ g_{33} \in C^1_c (\R) .
\]
For such $\phi$ and $\vec g$ we have
\[
 D \phi (\vec x) = \binom{\phi_3 (x_3) \, D \phi_1 (\hat{\vec x})}{\phi_1 (\hat{\vec x}) \, \phi'_3 (x_3)} , \qquad \div \vec g (\vec y) = g_{13} (y_3) \div \vec g_{11} (\hat{\vec y}) + g_{31} (\hat{\vec y}) \, g'_{33} (y_3) .
\]
On the other hand, for a.e.\ $\vec x \in \omega \times I$
\[
 D \vec w (\vec x) = \begin{pmatrix}
 D \vec v (\hat{\vec x}) & \vec 0 \\
 \vec 0 & 1
 \end{pmatrix} , \ 
  \cof D \vec w (\vec x) = \begin{pmatrix}
 \cof D \vec v (\hat{\vec x}) & \vec 0 \\
 \vec 0 & \det D \vec v (\hat{\vec x})
 \end{pmatrix} , \ 
 \det D \vec w (\vec x) = \det D \vec v (\hat{\vec x}) .
\]
Therefore, for a.e.\ $\vec x \in \omega \times I$,
\begin{equation}\label{eq:integrandE}
\begin{split}
 & \vec g( \vec w(\vec x))\cdot \left( \cof D \vec w(\vec x) D\phi(\vec x) \right) + \phi( \vec x)\div \vec g ( \vec w(\vec x)) \det D \vec w(\vec x) \\
 & = \phi_3 (x_3) \, g_{13} (x_3) \left[ \vec g_{11} (\vec v (\hat{\vec x})) 
 	\cdot
 \Big (\cof D \vec v (\hat{\vec x}) \, D \phi_1 (\hat{\vec x})\Big )
  + \phi_1 (\hat{\vec x}) \div \vec g_{11} (\vec v (\hat{\vec x})) \det D \vec v (\hat{\vec x}) \right] \\
 & \quad + g_{31} (\vec v (\hat{\vec x})) \, \phi_1 (\hat{\vec x}) \det D \vec v (\hat{\vec x}) \left[ g_{33} (x_3) \, \phi'_3 (x_3) + \phi_3 (x_3) \, g_{33}' (x_3) \right] .
\end{split}
\end{equation}
Now, since $\E (\vec v) = 0$ and $\E (\id_I) = 0$ we have that
\[
 \int_{\omega} \left[ \vec g_{11} (\vec v (\hat{\vec x}))
 \cdot
 \Big ( \cof D \vec v (\hat{\vec x})\, D \phi_1 (\hat{\vec x}) \Big )
  + \phi_1 (\hat{\vec x}) \div \vec g_{11} (\vec v (\hat{\vec x})) \det D \vec v (\hat{\vec x}) \right] \dd \hat{\vec x} = 0
\]
and
\[
 \int_I \left[ g_{33} (x_3) \, \phi'_3 (x_3) + \phi_3 (x_3) \, g_{33}' (x_3) \right] \dd x_3 = 0
\]
so an integration of \eqref{eq:integrandE} in $\omega \times I$ yields $\E_{\vec w} (\phi, \vec g) = 0$.
\end{proof}

In the following, we show that the precomposition of a map of zero energy with a regular map has also  zero energy; 
a related result was shown in the proof of \cite[Th.\ 7]{HeMo11}.

\begin{lemma}\label{le:ELu}
Let $\vec u\in H^1(\O,\R^3) \cap L^{\infty} (\O,\R^3)$ be such that $\det D \vec u \in L^1(\O)$ and $\E(\vec u) = 0$.
Let $\vec L : \R^3 \to \R^3$ be locally Lipschitz.
Then $\E (\vec L \circ \vec u) = 0$.
\end{lemma}
\begin{proof}
First assume that $\vec L$ is of class $C^2$.
Let $\phi \in C^1_c (\O)$ and $\vec g \in C^1_c (\R^3, \R^3)$.
The key of the proof consists in showing that
\begin{equation}\label{eq:ELu}
 \bar{\E}_{\vec L \circ \vec u} (\phi, \vec g) = \bar{\E}_{\vec u} (\phi, (\adj D \vec L) (\vec g \circ \vec L)) .
\end{equation}
The integrand corresponding to $\bar{\E}_{\vec L \circ \vec u} (\phi, \vec g)$ is, for a.e.\ $\vec x \in \O$,
\begin{equation}\label{eq:ELufg}
\begin{split}
 & \vec g ( (\vec L \circ \vec u) (\vec x)) \cdot \left( \cof D (\vec L \circ \vec u) (\vec x) \, D\phi(\vec x) \right) + \phi( \vec x) \div \vec g ( (\vec L \circ \vec u)(\vec x)) \det D (\vec L \circ \vec u) (\vec x) \\
 & = \left( \adj D \vec L (\vec u (\vec x)) \, \vec g ( \vec L (\vec u (\vec x))) \right) \cdot \left( \cof D \vec u (\vec x) )\, D\phi(\vec x) \right) \\
 & \quad + \phi( \vec x) \div \vec g ( \vec L (\vec u (\vec x))) \det D \vec L (\vec u (\vec x)) \det D \vec u(\vec x) .
\end{split}
\end{equation}
Now define $\bar{\vec g} := (\adj D \vec L) (\vec g \circ \vec L)$.
Then, for all $\vec y \in \R^3$, by Piola's identity,
\[
 \div \bar{\vec g} (\vec y) = \sum_{i, j=1}^3 \frac{\p}{\p y_i} \left[ \adj D \vec L (\vec y)_{ij} \, g_j (\vec L (\vec y)) \right] = \sum_{i, j=1}^3 \adj D \vec L (\vec y)_{ij} \, \frac{\p}{\p y_i} \left[ g_j (\vec L (\vec y)) \right] ,
\]
so, thanks to the matrix identity $\vec F \adj \vec F = (\det \vec F) \vec I$ valid for $\vec F \in \R^{3 \times 3}$,
\begin{align*}
 \div \bar{\vec g} (\vec y) & = \sum_{i, j=1}^3 \adj D \vec L (\vec y)_{ij} \, \frac{\p}{\p y_i} \left[ g_j (\vec L (\vec y)) \right] = \sum_{i, j, k=1}^3 \adj D \vec L (\vec y)_{ij} \, D \vec g (\vec L (\vec y))_{jk} \, D \vec L (\vec y)_{ki} \\
 & = \tr \left( D \vec L (\vec y) \, \adj D \vec L (\vec y) \, D \vec g (\vec L (\vec y)) \right) = \tr \left( \det D \vec L (\vec y) \, D \vec g (\vec L (\vec y)) \right) \\
 & = \det D \vec L (\vec y) \div \vec g (\vec L (\vec y)) .
\end{align*}
With this, we find that the integrand of $\bar{\E}_{\vec u} (\phi, \bar{\vec g})$ coincides with \eqref{eq:ELufg}, so \eqref{eq:ELu} is proved.
As $\E (\vec u) = 0$ we obtain that $\bar{\E}_{\vec L \circ \vec u} (\phi, \vec g) = 0$.

Now assume, as in the statement, that $\vec L$ is only locally Lipschitz, and take a sequence $\{ \vec L_n \}$ in $C^2 (\R^3, \R^3)$ such that $\vec L_n \to \vec L$ a.e., $D \vec L_n \to D \vec L$ a.e.\ and
\[
 \sup_{n \in \N} \left\| \vec L_n \right\|_{W^{1,\infty} (B(\vec 0, \|\vec u\|_{L^{\infty} (\O, \R^3)}))} < \infty .
\]
The existence of such approximating sequence follows from a classic result (see, e.g., \cite[Th.\ 6.6.1]{EvGa92}).
By the first part of the proof,
\begin{align*}
 0 = \int_{\O} \bigl[ & \left( \adj D \vec L_n (\vec u (\vec x)) \, \vec g ( \vec L_n (\vec u (\vec x))) \right) \cdot \left( \cof D (\vec u (\vec x) \, D\phi(\vec x) \right) \\
 & + \phi( \vec x) \div \vec g ( \vec L_n (\vec u (\vec x))) \det D \vec L_n (\vec u (\vec x)) \det D \vec u(\vec x) \bigr] \dd \vec x .
\end{align*}
Taking limits, we obtain that $\bar{\E}_{\vec L \circ \vec u} (\phi, \vec g) = 0$.
By Lemma \ref{lem:link_Ebar_E}, $\E (\vec L \circ \vec u) = 0$.
\end{proof}

We recall now the definition of the distributional Jacobian determinant.
\begin{definition}
Let \(\vec u\in H^1 (\Om,\R^3) \cap L^{\infty}(\Om,\R^3)\).
The distributional Jacobian $\Det D \vec u$ of \(\vec u\) is the distribution defined by
\[
\langle \Det D \vec u, \varphi \rangle := -\frac{1}{3}\langle \adj D \vec u \, \vec u, D \varphi \rangle = - \frac{1}{3}\int_\Om \adj D \vec u\, \vec u \cdot D \varphi , \qquad \varphi \in C^1_c(\Om) .
\]
\end{definition}

\subsection{The axisymmetric setting}\label{subse:axisymmetric}
In most of the paper we will work with axisymmetric (with respect to the $x_3$-axis) maps and domains, which are defined as follows. 
We say that the set $\Om \subset \R^3$ is axisymmetric if
\[
 \bigcup_{\vec x \in \O} \left( \p B_{\R^2} ((0,0), |(x_1, x_2)|) \times \{ x_3 \} \right) \subset \O .
\]
When we define
\begin{align*}
 \pi : \R^3 & \to [0, \infty) \times \R & \vec P : [0, \infty) \times \R \times \R &\to \R^3 \nonumber \\ 
 \vec x & \mapsto \left( |(x_1, x_2)|, x_3 \right) & (r, \t, x_3) &\mapsto (r \cos \t, r \sin \t, x_3) ,
\end{align*}
the axisymmetry of $\O$ is equivalent to the equality
\begin{equation}\label{eq:axi}
 \Om = \left\{ \vec P (r, \t, x_3) : \, (r,x_3) \in \pi(\Om), \, \t\in[0,2\pi) \right\} .
\end{equation}

Given an axisymmetric set $\O$, we say that $\vec u:\Om\to\R^3$ is axisymmetric if there exists $\vec v : \pi(\Om)\to [0, \infty) \times \R$ such that
\begin{multline}\label{eq:uv}
 (\vec u \circ \vec P) (r,\t,x_3) = \vec P \left( v_1 (r,x_3), \t , v_2(r,x_3) \right), \text{ i.e.}, \\
  \vec u(r\cos \theta,r\sin \theta,x_3)=v_1(r,x_3)(\cos \theta \vec e_1+\sin \theta \vec e_2) +v_2(r,x_3) \vec e_3
\end{multline}
for all $(r, x_3,\t) \in \pi(\Om)\times[0,2\pi)$.
We will say that $\vec v$ is the function corresponding to $\vec u$. This $\vec v$ is uniquely determined by $\vec u$. 
Note that if $\O$ and $\vec u$ are axisymmetric then so is $\vec u (\O)$.

Given $\d>0$, we define \(C_\delta\) as the (open, infinite, solid) cylinder of radius $\d$:
\begin{equation}\label{def:Cdelta}
C_\delta := \left\{ \vec P (r, \t, x_3) : (r, \t, x_3) \in [0,\delta)\times [0, 2\pi) \times \R \right\}.
\end{equation}
We have $\pi(C_\delta)=[0,\delta)\times \R$.

\subsection{Prescribing the boundary data}
	\label{se:bdryData}

As mentioned in the Introduction, we fix a smooth bounded open set $\widetilde{\Om}$ of $\R^3$ such that 
$\Om \Subset \widetilde{\Om}$ and consider an orientation-preserving $C^1$ diffeomorphism $\vec b$ from 
the closure of $\widetilde{\Omega}$ to $\R^3$.
We assume that $\Omega$, $\widetilde \Omega$, and $\vec b$ are axisymmetric.
The function $\vec u$, originally defined in $\O$, is extended to $\widetilde{\Om}$ by setting $\vec u =\vec b$ in $\Omega_D:= \widetilde \Omega \setminus \overline \Omega$.

We assume that the extension to $\widetilde \Omega$, still called $\vec u$, is in $H^1 (\widetilde{\Om}, \R^3)$, 
as stated in the definition \eqref{eq:defAs_introduction} of our function space $\mathcal A_s$.
Regarding the class $\mathcal A_s^r$ of regular maps defined in \eqref{eq:defAsr}, its definition is the set the maps in $\mathcal A_s$ having zero surface energy in the extended domain $\widetilde \Omega$:
\begin{equation*}
\Asr = \{\vec u \in \As: \mathcal E(\vec u)=0 \ \text{in}\ \widetilde \Omega\}.
\end{equation*}
This way we avoid cavitation at the boundary \(\p \O\).
Observe that the condition $\det D\vec u>0$ a.e.\ in  definition \eqref{eq:defAs_introduction} 
is satisfied for any $\vec u \in H^1(\widetilde \Omega, \R^3)$ such that $E(\vec u)\leq E(\vec b)$,
thanks to the blow-up behaviour \eqref{eq:explosiveH} of the neo-Hookean energy as the Jacobian vanishes.

\subsection{A family of good open sets}

Given a nonempty open set \(U \Subset\widetilde{\O}\) with a \(C^2\) boundary, we denote by \(d: \widetilde{\O} \rightarrow \R\) 
the signed distance function to \(\p U\) given by
\begin{equation*} 
d(\vec x):= \begin{cases}
\dist(\vec x, \p U) & \text{ if } \vec x \in U \\
0 & \text{ if } \vec x \in \p U \\
-\dist(\vec x,\p U) & \text{ if } \vec x \in \O \setminus \overline{U}
\end{cases}
\end{equation*}
and
\begin{equation}
	\label{eq:def_distance_level_sets}
U_t:= \{ \vec x \in \widetilde{\O} : \ d(\vec x) >t \},
\end{equation}
for each \(t \in \R\). It is a classical result (see e.g.\ \cite{Dieudonne72}) that there exists \( \delta>0\) such that for 
all \( t \in (-\delta,\delta)\), the set \(U_t\) is open, compactly contained in \(\widetilde{\O}\) and has a \(C^2\) boundary.

\begin{definition}\label{def:good_open_sets}
Let \(\vec u \in H^1(\widetilde{\O},\R^3)\) be such that $\det D\vec u>0$ a.e. 
We define \( \mathcal{U}_{\vec u} \) as the family of nonempty open sets \(U \Subset \widetilde{\O}\) with a \(C^2\) boundary that satisfy
\begin{enumerate}[a)]
\item \( \vec u_{| \p U} \in H^1(\p U,\R^3)\), and \( (\cof \nabla \vec u)_{|\p U}\in L^1(U,\R^{3\times 3})\),
\item \(\p U \subset \O_0\), \(\mathcal{H}^2\)-a.e., where \(\O_0\) is the set in Definition \ref{def:Om0}, and \( (\nabla \vec u_{|\p U})(\vec x)= \nabla \vec u(\vec x)_{| T_{\vec x} \p U}\) for \( \mathcal{H}^2\)-a.e.\ \(\vec x\in \p U\),
\item \( \lim_{ \e \to 0} \int_0^\e \left| \int_{\p U_t} |\cof \nabla \vec u| \dd \mathcal{H}^2-\int_{\p U} |\cof \nabla\vec u| \dd \mathcal{H}^2 \right| \dd t=0,\)
\item For every \( \vec g \in C^1(\R^3,\R^3)\) with \( (\adj D\vec u)( \vec g \circ \vec u) \in L^1_{\text{loc}}(\widetilde{\O},\R^3)\),
\begin{align*}
\lim_{\e \rightarrow 0} \int_0^\e \Bigg| \int_{\p U_t} \vec g( \vec u (\vec x))\cdot 
&
	(\cof \nabla \vec u(\vec x) \vec \nu_t (\vec x)) \, \dd \mathcal{H}^2
\\  
	&
	-\int_{\p U} \vec g( \vec u (\vec x))\cdot (\cof \nabla \vec u(\vec x) \vec \nu (\vec x)) \, \dd \mathcal{H}^2 \Bigg| \, \dd t =0
\end{align*}
where \( \vec \nu_t\) denotes the unit outward normal to \(U_t\) for each \(t\in (0,\e)\), and \(\vec \nu\) the unit outward normal to \( U\).
\end{enumerate}
\end{definition}

Since we are imposing that $\vec u$ coincides with the $C^1$-diffeomorphism $\vec b$ in the exterior Dirichlet neighbourhood 
$\Omega_D$ of $\partial \Omega$, without loss of generality we may assume that $\O \in \mc{U}_{\vec u}$.

The following result is obtained using the coarea formula and Lebesgue's differentiation theorem for the first part, 
and Uryshon functions and Sard's lemma for the second. It guarantees that there are enough sets in $\mathcal{U}_{\vec u}$.

\begin{lemma}[Lemma 2.11 in \cite{HeMo15}]
		\label{le:enough_good_open_sets}
	Let $\vec u\in H^1(\widetilde \Omega, \R^3)$
	be such that $\det D\vec u>0$ a.e.
	Let 
	$U\Subset \widetilde \Omega$ be a
	nonempty open set with a $C^2$ boundary.
	Then
	$U_t\in \mathcal{U}_{\vec u}$ for a.e.\ $t\in (-\delta, \delta)$. 
	Moreover, for each compact $K\subset \widetilde \Omega$ there exists $U\in \mathcal{U}_{\vec u}$
	such that $K\subset U$.
\end{lemma}

\subsection{Regularity, injectivity and weak convergence of the planar function}

Let $C_\delta$ be as in \eqref{def:Cdelta}.
The link between the regularity of $\vec u$ and its associated \(2D\) map $\vec v$ is as follows.

\begin{lemma}\label{le:uvSobolev}
Let \(\Om\) be an axisymmetric domain, and $\vec u : \O \to \R^3$ an axisymmetric map with corresponding function $\vec v$.
Let $\d>0$.
Then, $\vec u \in H^1 (\Om \setminus \overline{C}_{\d} , \R^3)$ if and only if $\vec v \in H^1 (\pi(\Om) \setminus ([0, \d] \times \R ), \R^2)$.
Moreover, in this case,
\[
 \left\| \vec u \right\|_{H^1 (\Om \setminus \overline{C}_{\d} , \R^3)}^2 \leq 2 \pi \max \{ \|\vec x \|_{L^{\infty} (\O, \R^3)} , \d^{-1} \}  \left\| \vec v \right\|_{H^1 (\pi(\Om) \setminus ([0, \d] \times \R), \R^2)}^2 ,
\]
\[
 \left\| \vec v \right\|_{H^1 (\pi(\Om) \setminus ([0, \d] \times \R) , \R^2)}^2 \leq (2 \pi \d )^{-1} \left\| \vec u \right\|_{H^1 (\Om \setminus \bar{C}_{\d}, \R^3)}^2
\]
and 
for a.e.\ $(r, \t, x_3)$ with $(r, x_3) \in \pi(\Om) \setminus ([0, \d] \times \R )$ and $\t \in \R$,
\begin{equation*} 
 \det D\vec u (\vec P (r, \t, x_3))= \frac{v_1 (r, \t)}{r}\det D \vec v (r, \t) .
\end{equation*}
\end{lemma}

This can be proved by using that the change of variables from Cartesian to cylindrical coordinates is a diffeomorphism when restricted to 
suitable domains and by using the formula for the Dirichlet energy in cylindrical coordinates given in the Appendix.

The orientation-preserving and injectivity conditions of $\vec u$ and $\vec v$ are related as follows.
We recall that a function is injective a.e.\ if its restriction to a set of full measure is injective.

\begin{lemma}\label{le:uvinj}
Let \(\Om\) be an axisymmetric domain.
Let $ \vec u\in H^1 (\Om, \R^3)$ be axisymmetric, and let $\vec v$ be its corresponding function.
The following hold:
\begin{enumerate}[a)]
\item $\det D \vec u > 0$ a.e.\  in \(\O\) if and only if $\det D \vec v > 0$ a.e.\  in \( \pi(\O)\).

\item $\vec u$ is injective a.e.\  in \(\O\) if and only if $\vec v$ is injective a.e\@ in \( \pi(\O)\).
\end{enumerate}
\end{lemma}

The proof uses again the change of variables in cylindrical coordinates and is left to the reader.

The relationship between the weak convergences (hereafter denoted by $\weakc$) of a sequence of axisymmetric functions and their associated functions is contained 
in the two following lemmas, whose proofs are left again to the reader since they just rely on a manipulation of the axisymmetry in cylindrical coordinates.

\begin{lemma}\label{le:weakconv}
Let \(\Om\) be an axisymmetric domain.
For each $j \in \N$, let $\vec u_j \in H^1(\Om, \R^3)$ be axisymmetric.
Let $\vec u \in H^1 (\Om, \R^3)$, and assume that $\vec u_j \weakc \vec u$ in $H^1(\Om, \R^3)$ as $j \to \infty$.
The following statements hold:
\begin{enumerate}[i)]
\item $\vec u$ is axisymmetric.

\item\label{item:wcii} Let $\vec v_j$ and $\vec v$ be the corresponding functions of $\vec u_j$ and $\vec u$, respectively.
Then $\vec v_j \weakc \vec v$ in $H^1 (\pi(\Om) \setminus ([0, \d] \times \R ), \R^2)$ for each $\d >0$.
\end{enumerate}
\end{lemma}

\begin{lemma}
Let \(\Om\) be an axisymmetric domain.
For each $j \in \N$, let $\vec u_j : \O \to \R^3$ be an axisymmetric map with corresponding function $\vec v_j$.
Let $\vec v : \pi(\Om) \setminus (\{ 0 \} \times \R) \to \R^3$ and assume that $\vec v \in H^1 (\pi(\Om) \setminus ([0, \d] \times \R ), \R^3)$ and that $\vec v_j \weakc \vec v$ in $H^1(\pi(\Om) \setminus ([0, \d] \times \R ), \R^3)$ for every $\d>0$.
Define $\vec u : \O \setminus \R\vec{e}_3 \to \R^3$ by \eqref{eq:uv}.
Then $\vec u_j \weakc \vec u$ in $H^1 (\O \setminus \overline{C}_{\d})$ for every $\d>0$.
\end{lemma}

\subsection{Regularity of maps in \(\As\) outside the axis of symmetry}
We recall that orientation-preserving $H^1$ maps in $2D$ are  continuous and satisfy Lusin's condition $(N)$ and the divergence identities. These properties are inherited by the $3D$ axisymmetric maps $\vec u$ in $\As$ away from the symmetry axis.
\begin{lemma}\label{N condition}
Let \(\Om\) be an axisymmetric domain. Let $ \vec u\in H^1 (\Om, \R^3)$ be axisymmetric and satisfy $\det D \vec u > 0$ a.e., and let $\vec v$ be its corresponding function.
Then:
\begin{enumerate}[a)]
\item\label{item:Auv} $\vec v$ has a representative that is continuous at each point of $\pi(\Om) \setminus (\{ 0\} \times \R)$, 
differentiable a.e., and satisfies condition $(\text{N})$ in $\pi(\Om)\setminus (\{0\}\times \R)$.
Moreover, $\E (\vec v, \pi(\Om)\setminus (\{0\}\times \R)) = 0$.
\item $\vec u$ has a representative that is continuous at each point of $\Om\setminus \R \vec e_3$, 
differentiable a.e., and satisfies condition $(\text{N})$ in $\Om \setminus \R \vec e_3$.
Moreover, $\E (\vec u, \O \setminus \R\vec{e}_3) = 0$.
\item\label{item:wciii}  For each $j \in \N$, let $\vec u_j \in H^1(\Om, \R^3)$ be axisymmetric.
Let $\vec u \in H^1 (\Om, \R^3)$, and assume that $\vec u_j \weakc \vec u$ in $H^1(\Om, \R^3)$ as $j \to \infty$. 
If $\det D \vec u_j > 0$ a.e.\ for all $j \in \N$ then $\vec v_j \to \vec v$ uniformly in compact subsets of 
$\pi(\Om) \setminus (\{ 0 \} \times \R )$ and $\vec u_j \to \vec u$ uniformly in compact subsets of $\Om\setminus (\{ (0, 0) \} \times \R )$.
\end{enumerate}

\end{lemma}
\begin{proof}
Let $\d>0$.
By Lemma \ref{le:uvSobolev}, $\vec v$ is in $H^1 (\pi(\Om)\setminus ([0, \d] \times \R),\R^2)$ and, by Lemma \ref{le:uvinj}, $\det D\vec v >0$ a.e\@.
By classical results on maps of finite distorsion, originally due to \cite{Reshetnyak67,VoGo76} 
(see also, e.g., \cite[Th.\ 2.5.4, Th.\ 5.3.5 and its Cors.\ 1 and 3]{GoRe90} or \cite[Th.\ 2.3, Cor.\ 2.25 and Th.\ 4.5]{HeKo14}), $\vec v$ has a 
representative $\overline{ \vec v}$ in $\pi(\Om) \setminus ( [0, \d] \times \R)$ 
that is continuous, differentiable a.e.\ and satisfies the $(\text{N})$ property.
That $\E (\vec v, \pi(\Om)\setminus ([0, \d]\times \R)) = 0$ is also a classical result (see, e.g., \cite{Muller88,Sverak88,Muller90CRAS,MuQiYa94}).
As this is true for every $\d>0$, property \emph{\ref{item:Auv})} is proved.

We define the representative $\bar{\vec u}$ of $\vec u$ through formula \eqref{eq:uv}, but changing $\vec u, \vec v$ by $\bar{\vec u}, \bar{\vec v}$, respectively.
As in Lemma \ref{le:uvinj}, we readily obtain that $\bar{\vec u}$ is continuous in $\Om\setminus \R\vec e_3$ and differentiable a.e\@. We now show the $(\text{N})$ property for $\bar{\vec u}$.
Let $A$ be a null set in $\Om\setminus \R \vec e_3$, and for each $\t \in \R$ define $A_{\t}:= \{ (r, x_3) \colon \vec P (r, \t, x_3) \in A \}$.
Since $\mc{L}^3 (\vec P^{-1} (A)) = 0$, we have that $\mc{L}^2 (A_{\t})= 0$ for a.e.\ $\t \in \R$.
For any such $\t$ we have that $\bar{\vec v} (A_{\t})$ is $\mc{L}^2$-null.
By \eqref{eq:uv}, $\vec P(s,\t,y_3)\in \bar{\vec u} (A)$ if and only if $(s,y_3)\in \bar{\vec v} (A_\t)$.
Consequently,
\begin{align*}
\mc{L}^3 (\bar{\vec u} (A))
=\int_{\R} \int_0^{2 \pi} \int_0^{\infty} \chi_{\bar{\vec u}(A)} (\vec P (s,\t,y_3)) \, s \, \dd s \, \dd\t \, \dd y_3
=\int_{\R} \int_0^{2 \pi} \int_0^{\infty} \chi_{\bar{\vec v} (A_\t)} (s,y_3) \, s \, \dd s \, \dd \t \, \dd y_3 .
\end{align*}
Therefore, for any $R>0$,
\begin{equation*}\begin{split}
\mc{L}^3 (\bar{\vec u} (A) \cap B(\vec 0, R)) \leq R \int_{\R} \int_0^{2 \pi} \int_0^{\infty} \chi_{\bar{\vec v} (A_\t)} (s,y_3) \, \dd s \, \dd\t \, \dd y_3 = R \int_0^{2\pi}\mc{L}^2(\bar{\vec v}(A_\t)) \, \dd\t = 0.
\end{split}\end{equation*}
As this is true for every $R>0$, we obtain that $\mc{L}^3 (\bar{\vec u} (A)) = 0$.
Thus, $\bar{\vec u}$ satisfies condition (N) in $\O \setminus \R \vec e_3$.

Let $I \subset \R$ be an open interval of length less than $2 \pi$ and define the function $\vec w$ in the set $\{ (r, \t, x_3) : (r, x_3) \in \pi (\O) \setminus ( \{0\} \times \R), \, \t \in I \}$ as $\vec w(r, \t, x_3) := (v_1 (r, x_3), \t, v_2 (r, x_3) )$.
As shown above, $\E (\vec v, \pi(\Om)\setminus (\{0\}\times \R)) = 0$, so by Lemma \ref{le:Evid}, $\E (\vec w) = 0$.
By Lemma \ref{le:ELu}, $\E (\vec P \circ \vec w) = 0$, so by \eqref{eq:uv}, $\vec u \circ \vec P$ has zero surface energy in $\{ (r, \t, x_3) : (r, x_3) \in \pi (\O) \setminus ( \{0\} \times \R), \, \t \in I \}$.
Let $\d > 0$.
As $\vec P |_{(\d, \infty) \times I \times \R}$ is a diffeomorphism that admits an extension to an open set containing the closure of $(\d, \infty) \times I \times \R$, by \cite[Sect.\ 8]{HeMo11} or \cite[Sect.\ 6]{HeRo18}, $\vec u$ has zero surface energy in $\Om \cap \vec P \left( (\d, \infty) \times I \times \R \right)$.
Considering now two open intervals $I_1$ and $I_2$ of length less than $2 \pi$, it is easy to check (see, if necessary, the proof of \cite[Lemma 4.8]{MoOl19}) that $\vec u$ has zero surface energy in $\Om \cap \vec P \left( (\d, \infty) \times (I_1 \cup I_2) \times \R \right)$.
Taking, additionally $I_1$ and $I_2$ such that $[0, 2\pi] \subset I_1 \cup I_2$, we obtain that $\vec u$ has zero surface energy in $\O \setminus \overline{C}_{\d}$.
Consequently, $\E (\vec u, \O \setminus \R\vec e_3) = 0$.

Now we show \emph{\ref{item:wciii})} and assume that $\det D \vec u_j > 0$ a.e.\ for all $j \in \N$.
By Lemma \ref{le:uvinj}, we have that $\det D \vec v_j > 0$ a.e.\ for all $j \in \N$.
Since the $H^1$ norm of $\{ \vec v_j \}_{j \in \N}$ is bounded in $\pi(\Om) \setminus ( [0, \d] \times \R )$ for each $\d>0$, by a classic result on maps of finite distortion (see, e.g., \cite[Lemma 2.1]{FoGa95} for a precise reference), the family $\{ \vec v_j \}_{j \in \N}$ is equicontinuous in each compact set of $\pi(\Om) \setminus ( [0, \d] \times \R )$, hence in each compact set of $\pi(\Om) \setminus (\{ 0 \} \times \R )$.
By the Ascoli--Arzel\`a theorem, and part \emph{\ref{item:wcii})} of Lemma \ref{le:weakconv}, $\vec v_j \to \vec v$ uniformly in each compact set of $\pi(\Om) \setminus (\{ 0 \} \times \R )$, in principle up to a subsequence, but in fact the convergence holds for the whole sequence because $\vec v$ is uniquely determined.
Having in mind \eqref{eq:uv}, we obtain that $\vec u_j \circ \vec P \to \vec u \circ \vec P$ uniformly in each compact subset of
\[
 \left\{ (r, \t, x_3 ) : \, (r, x_3) \in \pi (\O) \setminus (\{ 0 \} \times \R ) , \, \t \in \R \right\} .
\]
Then, as before, $\vec u_j \to \vec u$ uniformly in each compact subset of $\Om\setminus \{ (0, 0) \} \times \R$.
\end{proof}

The assumptions of Lemma \ref{N condition} will hold in most of the paper, and, in this case, without further mention, $\vec u$ and $\vec v$ are taken 
to be the continuous representative of themselves in the sets $\Om\setminus \R\vec e_3$ and $\pi(\Om) \setminus (\{ 0\} \times \R)$, respectively.

\section{Existence of minimizers of the neo-Hookean energy in the class $\mathcal{A}_s$}\label{sec:existence_as}

We prove in this section that, although the results of Ball \& Murat show that \(E\) is not  weakly lower semicontinuous in the full \(3D\) setting because of the phenomenon of cavitation, when restricted to the axisymmetric setting we can prove that \(E\) is weakly lower semicontinuous. The reason for that is that maps in \(\As\) are continuous outside the axis of symmetry, so cavitation can only occur on the axis of symmetry. Hence axisymmetric cavitation can be viewed as a cavitation on the boundary of the \(2D\) subdomain \(\pi(\O)\) and does not contradict $W^{1,2}$-quasiconvexity.

\begin{proposition}\label{prop:closedeness_of_Asym}
Let \( \{\vec u_n\}_n\) be a sequence in \(\As\). Then there exists \(\vec u \in \As\)  such that, 
up to a subsequence, \(\vec u_n \rightharpoonup \vec u\) in \(H^1(\widetilde{\Om},\R^3)\), 
\begin{equation*}
\det D \vec u_n \rightharpoonup \det D \vec u \text{ in } L^1(\widetilde\Om),
\end{equation*}
$\chi_{\imG(\vec u_n,\widetilde \Om)} \to \chi_{\imG(\vec u,\widetilde \Omega)}$ a.e.\ as $n \to \infty$, and
\begin{equation}\label{eq:lower-semicontinuity}
E(\vec u) \leq \liminf_{n\rightarrow \infty} E(\vec u_n).
\end{equation}
\end{proposition}
 
\begin{proof}
Since $E(\vec u_n)\leq E(\vec b)$ for all $n \in \N$, we have, thanks to \eqref{eq:E} and Poincar\'e's inequality together with the boundary condition $\vec u_n=\vec b$ in $\Omega_D$, that $\{ \vec u_n \}_{n \in \N}$ is bounded in $H^1(\widetilde{\Om},\R^3)$.
Thus there exists $\vec u\in H^1(\widetilde{\Om},\R^3)$ such that, 
up to a subsequence $\vec u_n \rightharpoonup \vec u$ in $H^1$ and \(\vec u_n \rightarrow \vec u\) a.e.
Therefore, 
\(\vec u=\vec b \) a.e.\ on \(\widetilde{\Om}\setminus \Om\).
Moreover, by Lemma \ref{le:weakconv}, $\vec u$ is axisymmetric. 
Besides, we have $\sup_n\int_{\tilde{\Om}} H(\det D \vec u_n) <+ \infty$. By using the De La Vall\'ee Poussin criterion, we 
find that there exists $ d \in L^1(\widetilde{\Om})$ such that 
\begin{equation*}
\det D \vec u_n \rightharpoonup d \ \text{ in } \ L^1(\widetilde{\Om}).
\end{equation*}
A standard argument based on \eqref{eq:explosiveH} and Fatou's lemma (see, e.g., \cite[Th.\ 5.1]{MuSp95}) shows that $d >0$ a.e.\ in $\widetilde{\Om}$.
Let \(\vec v_n\) the \(2D\) map associated to \(\vec u_n\). From \ref{le:uvinj} we have \(\det D\vec v_n>0\) a.e.\ in \(\pi(\widetilde{\O})\).
By Lemma \ref{le:weakconv} we also have \(\vec v_n\rightharpoonup \vec v\) in \(H^1_{\text{loc}}(\pi(\widetilde{\O}\setminus (\{0\}\times \R)),\R^2)\). 
We can thus apply the result about higher integrability of the Jacobians due to M\"uller \cite{Muller90} to obtain that \(\det D \vec v_n \rightharpoonup \det D \vec v\) 
in \(L^1_{\text{loc}}(\pi(\widetilde{\O}\setminus (\{0\}\times \R)))\). From Sobolev injections we find that \(v_1^n \rightarrow v_1\) in
 \(L^2_{\text{loc}}(\pi(\widetilde{\O}\setminus (\{0\}\times \R)))\) and a.e.\ (up to a subsequence).
Fix a small $\delta>0$. Since $\vec u=\vec b$ in $\widetilde \Omega \setminus \Omega$,
a consequence of Lemma \ref{N condition} applied to $\widetilde \Omega$ is that $\vec v$ is bounded and $(\vec v_n)$ is uniformly bounded in $\pi(\widetilde{\Omega})\setminus ([0,\delta]\times \R)$.
Also, in the axisymmetric setting we have that
 \begin{align*}
  \det D\vec u_n (\vec x) = \frac{1}{r}v_1^n(r,x_3)\det D\vec v_n (r, x_3)
 \end{align*}
 (see, e.g., \eqref{eq:comatrix_cylindrical_cylindrical} in the Appendix).
 Hence the use of Egorov's theorem, see e.g.\ \cite[Lemma 6.7]{SiSp00}, implies that
\begin{equation*}
\det D \vec u_n \rightharpoonup \det D \vec u \text{ in } L^1_{\text{loc}}(\widetilde{\O}\setminus L).
\end{equation*}
Since \(L\) has zero Lebesgue measure, we find that \(\det D\vec u=d>0\) a.e.\ in \(\widetilde{\O}\) and \(\det D \vec u_n \rightharpoonup \det D \vec u \text{ in } L^1(\widetilde{\O})\).

By Lemma \ref{N condition}, $\mc{E} (\vec u_n, \widetilde \O \setminus \R\vec e_3) = 0$ for all $n \in \N$.
Then, by \cite[Th.\ 2]{HeMo10}, 	 $\vec u$ is injective a.e. and for a.e.\ $\delta>0$ we have 
$\chi_{\imG(\vec u_n,\widetilde\Om\setminus \overline C_\delta)} \to \chi_{\imG(\vec u,\widetilde \Omega\setminus \overline C_\delta)}$ a.e.\ as $n \to \infty$.
From here, using the equiintegrability of the Jacobians, it is easy to prove that $\chi_{\imG(\vec u_n,\widetilde \Om)} \to \chi_{\imG(\vec u,\widetilde \Omega)}$ in 
$L^1(\widetilde \Omega)$. Passing to a subsequence we obtain the stated a.e.\ convergence.
		
Thanks to the weak continuity of the Jacobian and the convexity of $H$, we have that
\(E\) is sequentially lower semicontinuous for the weak convergence in \(H^1\), i.e., 
\eqref{eq:lower-semicontinuity} holds, and, in particular, $E(\vec u)\leq E(\vec b)$.
\end{proof} 

\begin{theorem}\label{th:main1}
The energy $E$ has a minimizer  in \(\As \).
\end{theorem}
\begin{proof}
It follows from the direct method of calculus of variations, since Proposition \ref{prop:closedeness_of_Asym} shows the weak lower semicontinuity of $E$
and the sequential weak compactness of $\As$.
\end{proof}

\section{Regularity of inverses of maps in $\As$}
	\label{se:fine_properties}

The goal of this section is to give the appropriate definition of the inverse of maps in \(\As\) and to prove that its first two components enjoy Sobolev regularity. 
This property is crucial in the proof of Proposition \ref{prop:F_lower_semi_cont}, which is the key for our main result Theorem \ref{th:main_theorem_introduction}.

\subsection{Topological degree}

We first recall how to define the classical Brouwer degree for continuous functions \cite{Deimling85,FoGa95book}. Let \(N \geq 2\). 
Let \( U \subset \R^N \) be a bounded open set. If  \(\vec u \in C^1(\overline{U},\R^N)\) 
then for every  regular value  \(\vec y\) of \(\vec u\) we set
\begin{equation}\label{def:degree_1}
\deg(\vec u,  U, \vec y)= \sum_{\vec x \in {\vec u}^{-1}(y) \cap U} \det D {\vec u}(\vec x).
\end{equation}
This sum is finite thanks to the inverse function theorem. We can show that the right-hand side of \eqref{def:degree_1} is invariant by homotopies. 
This allows to extend the definition \eqref{def:degree_1} to every \(\vec y \notin \vec u(\p U)\) and to show its
depends only on the boundary values of $\vec u$.
If \(\vec u\) is only in \( C(\p U,\R^N)\),
we may extend $\vec u$ to a continuous map in $\overline U$ by Tietze's theorem and set
\begin{equation*}
\deg(\vec u, U, \cdot)= \deg(\vec v, U ,\cdot),
\end{equation*}
where \(\vec v\) is any map in \(C^1(\overline U, \R^N)\) which is homotopic to the extension of \(\vec u\).

If $U$ is of class $C^1$ and \(\vec u \in C^1(\p U,\R^N)\), by using \eqref{def:degree_1}, Sard's theorem and the divergence 
identities \eqref{eq:divergence_identities}, we can make a change of variables and integrate by parts to obtain
\begin{equation}\label{def:degree_integral}
\int_{\R^N} \deg(\vec u ,U,\vec y) \div \vec g(\vec y) \, \dd \vec y = \int_{\p U} (\vec g \circ \vec u) \cdot \left( \cof D \vec u \, \vec \nu \right) \dd \mathcal{H}^{N-1}.
\end{equation}
This formula can be used as the definition of the degree for maps in \(W^{1,N-1} \cap L^\infty(\p U,\R^N)\) as noticed by Brezis \& Nirenberg \cite{Brezis_Nirenberg_1995}.
For any open set $U$ having a positive distance away from the symmetry axis $\R^3$ it is possible to use the classical degree
since there every map in $\mathcal A_s$ has a continuous representative (Lemma \ref{N condition}).
However, for open sets $U$ crossing the axis (where maps in $\mathcal A_s$ may have singularities) we use the Brezis--Nirenberg degree.
 
\begin{definition}\label{prop:degree}
Let $U\subset \R^N$ be a bounded open set. For any $\vec u \in C(\partial U, \R^N)$ and any $\vec y \in \R^N\setminus \vec u(\partial U)$
we denote by $\deg(\vec u,  U, \vec y)$ the classical topological degree of $\vec u$ with respect to $\vec y$.
Suppose now that \(U \subset \R^N\) is a \(C^1\) bounded open set and \(\vec u \in W^{1,N-1}(\p U,\R^N)\cap L^\infty(\p U,\R^N)\). 
Then the degree of \(\vec u\), denoted by \( \deg( \vec u, U, \cdot)\), is defined as the only \(L^1\) function that satisfies \eqref{def:degree_integral}
for all \( \vec g \in C^\infty(\R^N,\R^N)\).
\end{definition}

To see that this definition makes sense we refer to \cite{Brezis_Nirenberg_1995} or \cite[Rk.\ 3.3]{CoDeLe03}. Also, using  \eqref{def:degree_integral} 
for a sequence of smooth maps approximating $\vec u$ we can see that for any 
\(\vec u \in C(\p U, \R^N)\cap W^{1,N-1}(\p U,\R^{N-1})\) such that $\mathcal L^N\big ( \vec u(\partial U)\big )=0$
the two definitions are consistent (as stated in \cite[Prop.~2.1.2]{MuSp95}).
Thanks to Lemma \ref{N condition}, our maps $\vec u$ satisfiy the (N) property, so the condition on $\vec u(\partial U)$
is satisfied for all regular open sets $U$ that are a distance apart from the axis.
On the one hand, by the continuity property of the degree we can see that the topological image (Definition \ref{de:imT} below) of a bounded open set is open.
This will give us adequate ambient spaces to work with in the deformed configuration, see Equation \eqref{eq:Om'}.  
On the other hand, it is by working with the Brezis--Nirenberg degree that the Sobolev regularity of the inverses 
(crucial for the lower semicontinuity result presented in this paper) is obtained in the presence of singularities (Proposition \ref{INV2}).

We will invoke many previous results about the degree and related concepts (such as the topological image or the condition INV; see below).
Most of the references that we cite use the degree with slightly different assumptions on $\vec u$.
Nevertheless, their proofs apply to our case with only small modifications.

\subsection{Topological image for the classical degree}\label{top image classic}

An important part of our analysis refers only to open sets $U$
that either are a distance apart from the symmetry axis or enclose entirely the closed segment  
\begin{equation}\label{def:segment_L}
L:= \overline{\O}  \cap \R \vec e_3
\end{equation}
where the singularities can occur.
To be precise, we use the setting of Section \ref{se:bdryData} and 
frequently deal with open sets $U \subset \widetilde{\O}$ such that 
$\partial U \cap L=\varnothing$.
Since $\vec u=\vec b$ in the Dirichlet region $\Omega_D=\widetilde \Omega\setminus \overline \Omega$, 
the map $\vec u$ is continuous in $\p U$ and, hence, the classical degree $\deg (\vec u, U, \cdot)$ is well defined.
For those sets $U$ the topological image is defined as follows.

\begin{definition}\label{de:imT}
Let \(N \geq 2\). Let $U$ be a  bounded open set of $\R^N$ and let $\vec u \in C(\p U,\R^N)$.
We define $\imT(\vec u,U)$, the topological image of $U$ under $\vec u$, as the set of $\vec y\in \R^N \setminus \vec u (\p U)$ such that $\deg( \vec u, U ,\vec y) \neq 0$.
\end{definition}

In the $2D$ case, the topological image through an orientation-preserving $H^1$ map enjoys some nice geometric properties.
Let $A$ be a bounded domain of $\R^2$ and let $\vec v\in H^1(A,\R^2)$
be a map such that $\det D \vec u > 0$ a.e. in $A$. As recalled in the proof of Lemma \ref{N condition}, $\vec v$ has a continuous representative. In what follows, we
identify $\vec v$ with that representative. Let $V$ be compactly included in $A$. We assume that  $V\in \mathcal{U}_{\vec v}$, i.e., all the analogous properties of Definition \ref{def:good_open_sets} are satisfied for the planar map $\vec v$.
In \cite[Lemma 5.4]{BaHeMo17} it was proved the following result:
\begin{equation}\label{eq:bdryimT}
\overline{\imT(\vec v, V)} = \imT(\vec v, V) \cup \vec v(\partial V) \quad \text{and} \quad \p \imT(\vec v, V)=\vec v(\partial V).
\end{equation}

Define (cf.~\cite[Def.~5.6]{BaHeMo17}) the topological image of a point $\vec x$ as
\begin{equation*}
\imT\big ( \vec v, \vec x \big )
:= \bigcap_{\rho>0 , \: B(\vec x, \rho) 
\in \mathcal{U}_{\vec v}}
\overline{\imT\big (\vec v, B(\vec x, \rho)\big )}.
\end{equation*}	
Since our map $\vec v$ is continuous, we have the characterization
$$\imT\big ( \vec v, \vec x\big ) =
\{\vec v(\vec x)\} \  \text{for every } \vec x\in V.$$
Indeed, by \cite[Cor.\ 1]{Sverak88}, $\vec v(B(\vec x,\rho))$ is included a.e.\ in $\overline{\imT(\vec v, B(\vec x,\rho)) }$, 
for each $\rho > 0$ with $B(\vec x, \rho) \in \mathcal{U}_{\vec v}$.
Therefore, since $\overline{\imT(\vec v, B(\vec x,\rho)) }$ is compact, the continuity implies that $\vec v(\vec x)$ belongs to 
$\overline{\imT(\vec v, B(\vec x,\rho)) }$ and therefore to 
$\imT\big ( \vec v, \vec x \big )$, since $\rho$ is arbitrary.
On the other hand, again by the continuity, $\imT( \vec v, \vec x)$ is a 
singleton as proved in \cite[Lemma 4]{Sverak88}.

Assume now $\vec v$ is injective a.e.\ in $A$. 
Let
\begin{equation*}
 G(\vec y):=\{\vec x\in \overline{V} \colon \vec v(\vec x)=\vec y\} \quad \text{and} \quad T:=\{\vec y\in \imT(\vec v, V) \colon G(\vec y) \text{ is not a singleton}\}.
\end{equation*}
Thus, $T$ is the image of the points where $\vec v$ is not injective. 
By \cite[Lemma 5.13 and Prop.\ 5.14]{BaHeMo17} we have that
if $\vec y\in\imT(\vec v,V)$, then $G (\vec y)\subset V$ and $G (\vec y)$ is connected.
Moreover, $\imT(\vec v,V)\subset \bigcup_{\vec x\in V}\imT( \vec v, \vec x)$ 
and $G (\vec y)\cap V \neq \varnothing$ for every $\vec y\in\imT(\vec v,V)$.
This means that whenever we take $\vec y\in \imT(\vec v,V)$, there is at least
an $\vec x\in V$ such that $\vec y\in \imT( \vec v, \vec x)$. 
By \cite[Th.\ 7]{Sverak88}, we have $\mathcal{H}^1(T)=0$.
We will use this last property later in order to define the inverse.
To sum up:  $\partial V$ is mapped by $\vec v$ onto 
$\partial\imT(\vec v, V)$, while $V$ is mapped in $\overline{\imT(\vec v, V)}$,
with the possibility that a point $\vec x$ is mapped to a point $\vec y\in \partial\imT(\vec v, V)$. However, when this happens, there exists a 
set $G(\vec y)\subset\overline{V}$ such that $\vec x\in G(\vec y)$ and 
$G(\vec y)\cap \partial V\neq\varnothing$. Roughly speaking, the deformation $\vec v$
may pinch an internal part of the domain to the boundary (see Figure \ref{fig1}).

\begin{figure}
\centering
\includegraphics[width=.61\textwidth]{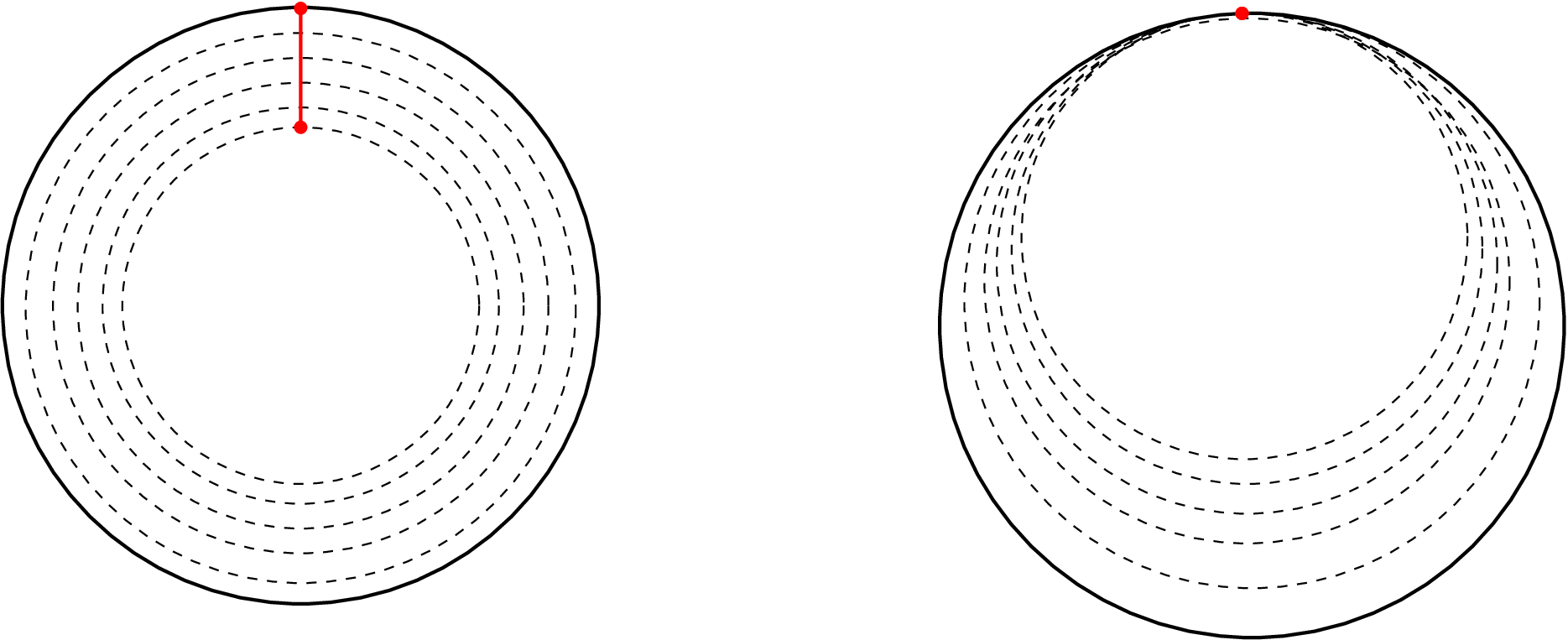}
\caption{On the left the sets $V$ and in red the set $G(\vec y)$, on the right
the set $\imT( \vec v, V)$ and in red the point $\vec y$ on the boundary.}
\label{fig1}
\end{figure}
\begin{lemma}\label{le:imTmonotone}
Let $A$ be a bounded domain of $\R^2$ and let $\vec v\in H^1(A,\R^2)$
be injective a.e.\ and such that $\det D \vec v > 0$ a.e. Moreover,
let $V\in  \mathcal{U}_{\vec v}$ and let be $U$ an open set such that $U\Subset V$. 
Then $\imT( \vec v, U)\subset\imT( \vec v, V)$.
\end{lemma}
\begin{proof}
By \cite[Lemma 6.2]{Reshetnyak89}, 
$\overline{\imT( \vec v, U)}\subset\overline{\imT( \vec v, V)}$.
Assume by contradiction that there exists $\vec y\in \imT( \vec v, U)$ such that $\vec y\in \partial\imT( \vec v, V)$. Let $\vec x_1\in U$ and $\vec x_2\in \partial V$
be such that $\vec v(\vec x_1)=\vec v(\vec x_2)=\vec y$. Since $\vec v$ is continuous 
and $\imT( \vec v, U)$, $\imT( \vec v, V)$ are open, there exist $\vec x'_1\in U$ and 
$\vec x'_2\in V\setminus\overline U$ such that $\vec v(\vec x'_1)=\vec v(\vec x'_2)=\vec y'$ for some $\vec y'\in \imT( \vec v, U)\cap \imT( \vec v, V)$ close to $\vec y$.
Then, since $G(\vec y')$ is connected, there exists $\vec x_3'\in \partial U\cap G(\vec y')$. Therefore, 
$\vec v(\vec x'_1)=\vec v(\vec x'_3)\in\partial\imT( \vec v, U)$, and the initial assumption must be false since $\imT( \vec v, U)$ is open. 
\end{proof}

Some of the properties of the topological image in the $2D$ case can be transposed to the axisymmetric setting, thanks to the following result.

\begin{lemma}\label{le:imTuimTv}
Let \(\Om \subset \R^3\) be an axisymmetric domain. Let $ \vec u\in H^1 (\Om, \R^3)$ be axisymmetric and satisfy $\det D \vec u > 0$ a.e., and let $\vec v$ be its corresponding function.
Let $U \subset \O$ be an axisymmetric  open set such that \(\overline{U} \cap \R \vec e_3 =\varnothing\).
Then $\imT(\vec v,\pi(U)) = \pi (\imT(\vec u, U))$.
\end{lemma}
\begin{proof}
Let $\vec z \in \R^2 \setminus \vec v (\p \pi(U))$ and $\t \in \R$.
Let $I$ be an open interval of length less than $2 \pi$ containing $\t$.
Set $U_ I := \{ (r, \t', x_3) : (r, x_3) \in \pi(U) , \, \t' \in I \}$.
By the product property for the degree (see, e.g., \cite[Th.\ 8.7]{Brown93}),
\begin{align*}
 \deg \left( (v_1, \id_{\R}, v_2 ) , U_ I  , (z_1, \t, z_2) \right) & = \deg \left( \vec v , \pi(U) , \vec z \right) \deg (\id_{\R}, I, \t) = \deg \left( \vec v , \pi(U) , \vec z \right) ,
\end{align*}
where $\id_{\R}$ is the identity map in $\R$.
As $\vec P : \overline{U}_{ I} \to \R^3$ can be extended to an orientation-preserving diffeomorphism in an open set containing $\overline{U}_{ I}$ 
we obtain by the composition formula for the degree (see, e.g., \cite[Th.\ 5.1]{Deimling85}) that
\begin{align*}
 \deg \left( (v_1, \id_{\R}, v_2 ) , U_ I , (z_1, \t, z_2) \right) & = \deg \left( \vec P \circ (v_1, \id_{\R}, v_2 ) , U_ I  , \vec P (z_1, \t, z_2) \right) \\
 & = \deg \left( \vec u \circ \vec P , U_ I  , \vec P (z_1, \t, z_2) \right) ,
\end{align*}
where in the last formula we have used \eqref{eq:uv}.
Applying again the composition formula we obtain
\begin{align*}
 \deg \left( \vec u \circ \vec P , U_ I  , \vec P (z_1, \t, z_2) \right) = \deg \left( \vec u , \vec P (U_ I ) , \vec P (z_1, \t, z_2) \right) .
\end{align*}
Altogether, we have shown that 
\begin{equation*}
 \deg \left( \vec v , \pi(U) , \vec z \right) = \deg \left( \vec u , \vec P (U_ I ) , \vec P (z_1, \t, z_2) \right) .
\end{equation*}
Now we show that
\begin{equation}\label{eq:excision}
 \vec P (z_1, \t, z_2) \notin \vec u \left( \overline{U} \setminus \vec P (U_ I ) \right) .
\end{equation}
Indeed, let $\vec x \in \overline{U} \setminus \vec P (U_ I )$.
Then $\vec x = \vec P ( r \cos \t', r \sin \t', x_3)$ with $(r, x_3) \in \pi (\overline{U})$ and $\t' \in \R \setminus (I + 2 \pi \Z)$.
By \eqref{eq:uv}, $\vec u (\vec x) = \vec P (v_1 (r, x_3), \t', v_2 (r, x_3))$, which implies \eqref{eq:excision}.
In turn, \eqref{eq:excision} and the excision property of the degree (see, e.g., \cite[Th.\ 3.1]{Deimling85} or \cite[Th.\ 8.4]{Brown93}) yield
\[
 \deg \left( \vec u , \vec P (U_ I ) , \vec P (z_1, \t, z_2) \right) = \deg \left( \vec u , U , \vec P (z_1, \t, z_2) \right) ,
\]
which, together with \eqref{eq:excision} shows that
\begin{equation}\label{eq:deguvII}
 \deg \left( \vec v , \pi(U) , \vec z \right) = \deg \left( \vec u , U , \vec P (z_1, \t, z_2) \right) .
\end{equation}
Recapitulating, we have shown that if $\vec z \in \R^2 \setminus \vec v (\p \pi(U))$ and $\t \in \R$ then $\vec P (z_1, \t, z_2) \notin \vec u (\p U)$ 
and formula \eqref{eq:deguvII} holds.

Now let $\vec z \in \imT(\vec v,\pi(U))$.
Then $\vec z \notin \vec v (\p \pi(U))$ and $\deg \left( \vec v , \pi(U) , \vec z \right) \neq 0$.
By \eqref{eq:deguvII}, $\deg \left( \vec u , U , \vec P (z_1, \t, z_2) \right) \neq 0$ for any $\t \in \R$, so $ \vec P (z_1, \t, z_2) \in \imT(\vec u, U)$.
Consequently, $\vec z = (\pi \circ \vec P) (z_1, \t, z_2) \in \pi \left( \imT(\vec u, U) \right)$.

To prove the converse inclusion, we start with the following simple facts:
\begin{enumerate}[i)]
\item\label{item:degi} If $(r, x_3) \in \pi (U)$ and $\t \in \R$ then $\vec P (r, \t, x_3) \in U$.

\item\label{item:degii} If $(r, x_3) \notin \pi (U)$ and $\t \in \R$ then $\vec P (r, \t, x_3) \notin U$.

\item\label{item:degiii} If $(r, x_3) \in \p \pi (U)$ and $\t \in \R$ then $\vec P (r, \t, x_3) \in \p U$.

\end{enumerate}
Property \ref{item:degii}) is obvious, while \ref{item:degi}) is a consequence of the axisymmetry of $U$.
Let us show \ref{item:degiii}).
Let $(r, x_3) \in \p \pi (U)$ and $\t \in \R$.
Note that $\pi (U)$ is an open set in \(\R^2\), so
\[
 \p \pi (U) = \overline{\pi (U)} \setminus \pi (U) .
\]
As $(r, x_3) \in \overline{\pi (U)}$, an elementary argument based on \ref{item:degi}) and the continuity of $\vec P$ shows that $\vec P (r, \t, x_3) \in \overline{U}$.
On the other hand, property \ref{item:degii}) shows that $\vec P (r, \t, x_3) \notin U$, so $\vec P (r, \t, x_3) \in \p U$ and \ref{item:degiii}) is proved.
 
We are in a position to show the inclusion $\pi (\imT(\vec u, U)) \subset \imT(\vec v,\pi(U))$.
Let $\vec z \in \pi (\imT(\vec u, U))$.
Then there exists $\t \in \R$ such that $\vec P (z_1, \t, z_2) \in \imT(\vec u, U)$.
Therefore,
\begin{equation}\label{eq:PimT}
 \vec P (z_1, \t, z_2) \notin \vec u (\p U) \text{ and } \deg \left( \vec u , U , \vec P (z_1, \t, z_2) \right) \neq 0 .
\end{equation}
We shall show that $\vec z \notin \vec v (\p \pi(U))$ by assuming that $\vec z = \vec v (r, x_3)$ for some $(r, x_3) \in \overline{\pi (U)}$.
Due to \eqref{eq:uv},
\[
 \vec P (z_1, \t, z_2) = \vec P (v_1 (r, x_3), \t, v_2 (r, x_3)) = \vec u \circ \vec P (r, \t, x_3) ,
\]
so, by \eqref{eq:PimT}, $\vec P (r, \t, x_3) \notin \p U$, and, by \ref{item:degiii}), $\vec z \notin \vec v (\p \pi(U))$.
Thanks to \eqref{eq:deguvII} and \eqref{eq:PimT}, $\deg \left( \vec v , \pi(U) , \vec z \right) \neq 0$, so $\vec z \in \imT(\vec v,\pi(U))$.
\end{proof}

\subsection{Topological image of the singular segment}

Throughout this section assume that $\vec b$, $\Omega$, $\widetilde \Omega$, and $\Omega_D$ are as in Section \ref{se:bdryData}.
Note that all $\vec u\in \As$ satisfy the properties stated in Lemma \ref{INV3}.

Away from the segment $L=\overline \Omega \cap \R\vec e_3$ condition INV is defined as follows.

\begin{definition}\label{def:INV}
Let $U$ be a bounded open set in 
$\R^3$. 
If $\vec u\in C(U,\R^3)$, we say that $\vec u$ satisfies property (INV) 
in $U$ provided that for every point \(\vec x_0 \in U\) and a.e.\ \(r\in (0,\dist(\vec x_0,\p U))\):
\begin{enumerate}[(a)]
\item \( \vec u(\vec x) \in \imT(\vec u,B(\vec x_0,r))\) for a.e.\ \(\vec x \in B(\vec x_0,r)\)
\item \(\vec u(\vec x) \notin \imT (\vec u,B(\Vec x _0,r))\) for a.e.\ \(\vec x \in \widetilde \O \setminus B(\vec x_0,r)\).
\end{enumerate}
\end{definition}

The degree of any map $\vec u$ in $\As$ with respect to any open set $U$ separated from the symmetry axis
coincides a.e.\  with the number of preimages (at which $\vec u$ is approximately differentiable) by $\vec u$ 
in that open set. This is shown now and relies on the fine regularity properties satisfied 
away from the axis and on the preservation of orientation.

\begin{lemma}
		\label{INV3}
Suppose that $\vec u \in H^1(\widetilde \Omega, \R^3)$ is axisymmetric, satisfies $\det D\vec u>0$ a.e.\ and $\vec u = \vec b$ in $\Omega_D$. 
Then:
	\begin{enumerate}[(a)]
	\item	
	$\vec u$ is continuous in $\widetilde \O \setminus L$
	and $\mathcal E(\vec u, \widetilde \O \setminus L)=0$.

	\item For any $U\in \mathcal{U}_{\vec u}$ (see Definition \ref{def:good_open_sets}) such that $\overline U \cap L =\varnothing$
	\begin{equation}
		\label{eq:deg_positive}
	\deg (\vec u, U, \cdot) = \mc{N}(\vec u,\O_d \cap U, \cdot) \ \text{ a.e.}
	\quad \text{and}\quad 
 \imG (\vec u, U) = \imT(\vec u, U)\ \text{a.e.},
	\end{equation}
	where $\Omega_d$ is the set of approximate differentiability.
	
	\item $\vec u$ satisfies condition (INV) in $\widetilde \Omega \setminus L$ if and only if $\vec u$ is injective a.e.\@ In particular, all maps in $\mathcal A_s$ satisfy (INV) in $\widetilde \Omega \setminus L$. 
	\end{enumerate}
\end{lemma}

\begin{proof}
Part (a) follows from Lemma \ref{N condition} and \cite[Lemma 4.8]{MoOl19}. Parts (b) and (c) can be obtained with the same proof of \cite[Th.\ 4.1 and Lemma 5.1.(a)]{BaHeMo17}.
\end{proof}

Recall that $\Om_{\vec b}:=\vec b(\Om)$ and $\widetilde \Om_{\vec b}:=\vec b(\widetilde \Om)$.

\begin{definition}\label{def:imT(u,L)}
Let \(\vec u \in \As \) and let 
$\mathcal{U}_{\vec u}^s:=\{U\in\mathcal{U}_{\vec u} \text{ is axisymmetric
and } U\Subset \widetilde \Omega \setminus \R \vec e_3\}$.
\begin{enumerate}[a)]
\item We define the topological image of $\widetilde{\O} \setminus L$ by $\vec u$ as

\[
 \imT (\vec u, \widetilde{\O} \setminus L) 
 :=\vec b(\O_D) \cup \bigcup_{\substack{U\in \mathcal{U}_{\vec u}^s}}
 	\imT (\vec u, U) .
\]

\item We define the topological image of $L$ by $\vec u$ as
\begin{equation*}
\imT(\vec u,L):=\widetilde \Om_{\vec b} \setminus \imT (\vec u,\widetilde{ \O} \setminus L) .
\end{equation*}
\end{enumerate}
\end{definition}

\begin{lemma}\label{le:contained_in_target_domain}
If $\vec u\in \mathcal A_s$, then: 
\begin{enumerate}[(a)]
\item
$\vec u(\vec x)\in \overline{\Omega_{\vec b}}$ for every $\vec x \in \Omega \setminus L$, 
\item 	
$\imT(\vec u, \widetilde \Omega \setminus L)
\subset \widetilde \Omega_{\vec b}$, and
\item\label{eq geo=top}
$ \imT (\vec u, \widetilde \O \setminus L) = \imG (\vec u, \widetilde \O \setminus L)=\imG(\vec u,\widetilde \O) \text{ a.e.}$
\end{enumerate}
\end{lemma}

\begin{proof}
Part (a): the map $\vec u$ cannot send a point $\vec x\in \O\setminus L$ in 
$\vec b(\O_D)$. Indeed, by continuity, a ball centered at $\vec x$ should be mapped into $\vec b(\O_D)$, against the fact that $\vec u$ is injective a.e.

Let us show that $\vec u$ cannot send a point $\vec x\in \O\setminus L$ 
outside $\widetilde \Omega_{\vec b}$. Observe that $\O_D$ wraps $\O$: for $\eta>0$ small enough, $\{\vec x \in \R^3 \setminus \overline{\Omega} 
\colon \dist(\vec x,\overline{\Omega}) < \eta\big \} \subset \O_D$.
Since $\vec b$ is a homeomorphism, $\vec b(\O_D)$ wraps $\O_{\vec b}$: for $\eta>0$ small enough, $\{\vec y \in \R^3 \setminus \overline{\Omega_{\vec b}} 
\colon \dist(\vec y,\overline{\Omega_{\vec b}}) < \eta\big \} \subset \vec b(\O_D)$.
Therefore, in order to  `exit' from $\Omega_{\vec b}$ one has to pass through 
$\vec b(\O_D)$.

Given $\vec x \in \Omega \setminus L$, let $\gamma$ be a continuous curve with endpoints $\vec x$ and $\vec x'\in \partial \O$ such that 
$\gamma\setminus\{\vec x'\}\subset\O \setminus L$. 
We have that $\vec u(\gamma)$ is connected since $\vec u$ is continuous in 
$\widetilde\Omega \setminus L$. 
Moreover, $\vec u(\vec x')\in \vec u(\partial\O)\Subset\widetilde \Om_{\vec b}$.
If $\vec u(\vec x)\notin\widetilde\Omega_{\vec b}$, then $\vec u(\gamma)$ has
to cross $\vec b(\O_D)$.
This implies that $\vec u$ maps at least one point $\vec x''\in \O$ in $\vec b(\O_D)$.
This contradicts what was proved at the beginning, finishing the proof of (a).
	
\smallskip	
Part (b): let $\vec v$ be the planar function corresponding to $\vec u$. Given $U\in \mathcal{U}_{\vec u}^s$, let $V\in \mathcal{U}_{\vec v}$ be such that 
and $\pi(U)\subset V\Subset\pi(\widetilde \Omega \setminus \R \vec e_3)$. 
Thanks to what was observed in Subsection \ref{top image classic}, $\imT(\vec v,V)\subset \vec v(V)$. Moreover, by Lemmas  \ref{le:imTmonotone} and \ref{le:imTuimTv} we have
$$\pi (\imT(\vec u, U))=\imT(\vec v,\pi(U)) \subset\imT(\vec v,V).$$
Therefore, by (a), $\imT(\vec u, U)\subset \widetilde \Omega_{\vec b}$.

\smallskip	
Part (c): for any $k\in\N$, let $U_k\in \mathcal{U}_{\vec u}^s$ be an axisymmetric open set containing
$$\{\vec x\in\widetilde \O 
\colon \dist(\vec x, \R \vec e_3)>1/k \}$$
and such that $V_k:=\pi(U_k)\in \mathcal{U}_{\vec v}$.
Any set $U\in\mathcal{U}_{\vec u}^s$ is included in some $U_k$.
Moreover, since by Lemma \ref{le:imTmonotone}
$\imT( \vec v, \pi(U))\subset\imT( \vec v, V_k)$, by Lemma \ref{le:imTuimTv} we have
$\imT(\vec u,U)\subset \imT(\vec u, U_k)$. This proves that
$$\bigcup_{\substack{U\in \mathcal{U}_{\vec u}^s}}
\imT (\vec u, U)=\bigcup_{k \in \N} \imT (\vec u, U_k).$$
By Lemma \ref{INV3}, $\imT (\vec u, U_k) = \imG (\vec u,U_k)$ a.e.\ for all $k\in\N$.
Hence, since $|L| = 0$,
\begin{equation*}\begin{split}
\imG(\vec u,\widetilde \O)
&\overset{\text{a.e.}}{=} \imG (\vec u, \widetilde \O \setminus L) = \vec b(\O_D) \cup\bigcup_{k \in \N} \imG (\vec u, U_k)\\
&\overset{\text{a.e.}}{=} 
\vec b(\O_D) \cup\bigcup_{k \in \N} \imT (\vec u, U_k)
=
\imT (\vec u, \widetilde \O \setminus L),
\end{split}\end{equation*}
where Lemma \ref{le:Lusin} was used in the first equality.
\end{proof}

Note that $\imT (\vec u, \widetilde{\O} \setminus L)$ is open as a union of open sets  and hence \(\imT( \vec u, L)\) is closed.
Also, by Lemma \ref{le:contained_in_target_domain}.(b),
\begin{align}
		\label{eq:created_surface_is_contained}
	\imT(\vec u, \widetilde \Omega \setminus L)
	= \widetilde \Omega_{\vec b} \setminus \imT(\vec u, L).
\end{align}
For example, in the construction of 
Conti \& De Lellis \cite{CoDeLe03}
$\imT(\vec u, L)$ consists (apart from the corresponding segment in the symmetry axis)
of the sphere 
$\partial B \big ( (0,0,\frac{1}{2}), \frac{1}{2} \big )$,
which may be regarded as new surface inside 
the elastic body created by the singular map $\vec u$.

\subsection{Sobolev regularity of inverses}\label{sob reg inverse}

As $\vec b : \tilde{\O} \to \R^3$ is a diffeomorphism, $\Om_{\vec b} = \imT(\vec b,\Om)$.
Moreover, for \(\vec u \in \As\), as $\vec u = \vec b$ in $\O_D$, 
the traces of $\vec u$ and $\vec b$ on $\p \O$ coincide.
As $\vec b |_{\p \O}$ is continuous, it is a representative of $\vec u|_{\p \O}$, hence the degree $\deg (\vec u, \O, \cdot)$ is defined and equals $\deg (\vec b, \O, \cdot)$.
In particular,
\begin{equation}\label{eq:Om'}
 \Om_{\vec b} =\vec b(\Om) = \imT(\vec b,\Om) = \imT(\vec u,\Om),
 \qquad
 \widetilde\Om_{\vec b}
 	= \imT(\vec u, \widetilde\Om).
\end{equation}

Let $\Omega_0$ be the set of Definition \ref{def:Om0}. 
Note that since maps in $\mathcal{A}_s$ are defined in $\widetilde \Omega$,
then $\Omega_0$ also contains points outside $\Omega$; in fact, it is of full measure in $\widetilde \Omega$.  
It was proved in \cite[Lemma 3]{HeMo11} that if $\vec u$ is one-to-one a.e.\ then $\vec u |_{\Omega_0}$ is injective.

\begin{definition}\label{inverse generic}
	Let $\vec u\in \mathcal A_s$. We 
	define its inverse as the map
	$\vec u^{-1}: \imG(\vec u, \Omega)\to \R^3$
	that sends every $\vec y\in \imG(\vec u, \Omega)$ to the only  $\vec x\in \Omega_0$ such that $\vec u(\vec x)=\vec y$. 
\end{definition}

Let $A$ be a bounded domain of $\R^2$ and let $\vec v\in H^1(A,\R^2)$
be injective a.e.\ and such that $\det D \vec v > 0$ a.e. 
From the comments in Subsection \ref{top image classic}, given $V\in  \mathcal{U}_{\vec v}$, we can define on $\imT(\vec v,V)$
$$\vec v^{-1} (\vec y) = \text{any element of } G (\vec y).$$
The definition is well posed; moreover, for every $\vec y\in \imT(\vec v,V)\setminus T$
we have
$\vec y= \vec v(\vec v^{-1} (\vec y))$.
Finally, by \cite[Lemma 6]{Sverak88} we have $\vec v^{-1}$ is continuous at any point of $\imT(\vec v,V)\setminus T$, 
while by \cite[Th.\ 8]{Sverak88} we have $\vec v^{-1}\in W^{1,1}(\imT(\vec v,V),\R^2)$ and
\begin{equation*}
 D \vec v^{-1} (\vec v (\vec x )) = D \vec v (\vec x)^{-1} = \frac{\adj D \vec v (\vec x)}{\det D \vec v (\vec x)}
 \quad \text{for a.e.\ } \vec x \in V.
\end{equation*}

All the properties above can be transposed to the axisymmetric case.
Let \(\Om\) be an axisymmetric domain, $ \vec u\in H^1 (\Om, \R^3)$ be axisymmetric, and let $\vec v$ be its corresponding planar function. 
Assume that $\vec u$ is injective a.e. and that $\det D \vec u > 0$ a.e.\  in \(\O\). By Lemma \ref{le:uvinj} the same
properties hold for $\vec v$ in $\pi(\O)$. Moreover, by Lemma \ref{le:uvSobolev}, if $U \subset \O$ is an axisymmetric  open set such 
that $\overline{U} \cap \R \vec e_3 =\varnothing$, then $\vec v\in H^1(\pi(U),\R^2)$.
Thanks to Lemma \ref{N condition}, we will consider a representative of $\vec u$ that is continuous at each point of $\Om\setminus \R \vec e_3$, 
and a representative of $\vec v$ that is continuous at each point of $\pi(\Om) \setminus (\{ 0\} \times \R)$. 
We have  $\pi(U)\subset V$ for some $V\in\mathcal{U}_{\vec v}$ (using the analogue of Lemma \ref{le:enough_good_open_sets} in $2D$). 
Let $\vec v^{-1}$ be the map defined in $\imT(\vec v,V)$ as above. Recalling that by Lemma \ref{le:imTuimTv}, 
$\imT(\vec v,\pi(U)) = \pi (\imT(\vec u, U))$, while by Lemma \ref{le:imTmonotone},
$\imT(\vec v,\pi(U))\subset\imT(\vec v,V)$, we can define on $\imT(\vec u, U)$
a map $\vec u^{-1}$ through formula \eqref{eq:uv}, changing $\vec u, \vec v$ by $\vec u^{-1}, \vec v^{-1}$, respectively. 
Let $R$ be the axisymmetric set such that $\pi(R)=T$; we have $\mathcal{H}^2(R)=0$ and for every $\vec y\in \imT(\vec u,U)\setminus R$ we have
$\vec y= \vec u(\vec u^{-1} (\vec y))$. 
Since by \eqref{eq:deg_positive}, $\imG (\vec u, U) = \imT(\vec u, U)\ \text{a.e.}$, 
if $\vec u\in \mathcal A_s$, then $\vec u^{-1}$ in $\imT(\vec u,U)$ is a specific representative of the inverse defined in Definition \ref{inverse generic}.
As in Lemma \ref{N condition}, we obtain that $\vec u^{-1}$ is continuous at any point of $\imT(\vec u, U)\setminus R$. Moreover, as in Lemma \ref{le:uvSobolev}, we obtain
that $\vec u^{-1}\in W_{\loca}^{1,1}(\imT(\vec u,U),\R^3)$ and
\begin{equation}\label{eq:formulaDu-1}
 D \vec u^{-1} (\vec u (\vec x )) = D \vec u (\vec x)^{-1} = \frac{\adj D \vec u (\vec x)}{\det D \vec u (\vec x)}
 \quad \text{for a.e.\ } \vec x \in U.
\end{equation}

\begin{lemma}\label{le:u-1}
Let $\vec u \in \As$. Then $\vec u^{-1} \in W^{1,1} (\widetilde \O_{\vec b} \setminus \imT(\vec u,L), \R^3)$ and formula \eqref{eq:formulaDu-1} holds for a.e.\ 
$\vec x \in \widetilde \O$.
\end{lemma}

\begin{proof}
Let $\{U_k\}_{k\in\N}$ be a sequence in $\mathcal{U}_{\vec u}^s$  as in the proof of Lemma \ref{le:contained_in_target_domain}.
From what we wrote before, $\vec u^{-1} \in W^{1,1} (\imT (\vec u, U_k), \R^3)$, and
this implies that $\vec u^{-1} \in W^{1,1}_{\loc}(\imT (\vec u, \widetilde \O \setminus L),\R^3)$. 
Using \eqref{eq:created_surface_is_contained}	we find that 
$\vec u^{-1} \in W^{1,1}_{\text{loc}} (\widetilde \O_{\vec b} \setminus \imT(\vec u,L), \R^3)$.	
Moreover, formula \eqref{eq:formulaDu-1} holds for a.e.\ $\vec x \in U_k$, from which we get immediately that it also holds for a.e.\ $\vec x \in \widetilde \O$.
Using that formula, and a change of variables,
we find that
\[
 \int_{\imT (\vec u, \widetilde \O \setminus L)} | D \vec u^{-1} (\vec y) | \, \dd \vec y = \int_{\widetilde \O} \left| \cof D \vec u (\vec x) \right| \dd \vec x < \infty
\]
since $\vec u \in H^1 (\widetilde \O, \R^3)$.
Therefore, 
$\vec u^{-1} \in W^{1,1}(\imT (\vec u, \widetilde \O \setminus L),\R^3)$.
\end{proof}

We now prove that when $\vec u$ has zero surface energy in $\widetilde \Omega$
then also the geometric image of $\widetilde \Omega$ (not only its topological image)
coincides with $\widetilde {\Omega}_{\vec b}$ (up to a Lebesgue-null set).
The first step for the proof is
to establish \eqref{eq:deg_positive}
also for open sets $U$ enclosing the singular segment $L$.

\begin{proposition}\label{INV}
Suppose that $\vec u \in H^1(\widetilde \Omega, \R^3)$ is axisymmetric 
and satisfies $\det D\vec u>0$ a.e., $\vec u = \vec b$ in $\Omega_D$,
and $\mathcal E(\vec u)=0$ in $\widetilde \Omega$. Then 
\begin{enumerate}[(a)]
\item
$\vec u\in L^\infty(\Omega,\R^3)$,
$\Det D \vec u = \det D \vec u$,
and $\vec u$ is injective a.e.
\item $\imG(\vec u, \widetilde \Omega) = {\widetilde \Omega}_{\vec b}$ a.e.
\item
For any $U\in \mathcal{U}_{\vec u}$,
	\begin{equation}\label{eq:imG=imT}
	\deg (\vec u, U, \cdot) = 
	\chi_{\imG(\vec u, U)} \ \text{ a.e.}
	\end{equation}
	In particular, when $\partial U \cap L=\varnothing$ (so that the classical degree and the topological image is well defined),
	\begin{equation} \label{eq:imG=imT2}
 \imG (\vec u, U) = \imT(\vec u, U)\ \text{a.e.}
	\end{equation}
\end{enumerate}
\end{proposition}

\begin{proof}
The results were proved in \cite[Th.\ 4.1]{BaHeMo17} for maps $\vec u \in W^{1,p}$ with $p>2$. 
Here we only explain the very minor modifications needed for the generalization to our $H^1$ setting. 
Now that we have $\mathcal E(\vec u)=0$ in $\widetilde \Omega$ (as opposed to Lemma \ref{INV3} where 
essentially any axisymmetric map was considered, so that $\mathcal E(\vec u)=0$ only in $\widetilde \Omega \setminus L$),
arguing exactly as in \cite[Th.\ 4.1]{BaHeMo17} we obtain both \eqref{eq:imG=imT2} and
\begin{equation}
		\label{eq:imG=imT99}
	\deg (\vec u, U, \cdot) =
	 \mc{N}(\vec u,\O_d \cap U, \cdot) \ \text{ a.e.},
	\end{equation}
where $\Omega_d$ is the set of approximate differentiability, for any $U\in \mathcal {U}_{\vec u}$
such that $\partial U \cap L =\varnothing$ (as opposed to Lemma \ref{INV3} where 
the stronger restriction $\overline U \cap L =\varnothing$ was imposed).
In particular, applying \eqref{eq:imG=imT99}  to any $U\in \mathcal{U}_{\vec u}$ such that $\Omega \Subset U \Subset \widetilde \Omega$, 
we find that for a.e.\ $\vec x \in \Omega_0$
$$
	\vec u(\vec x)\in \imG(\vec u, U) 
	\overset{\text{a.e.}}{=}\imT(\vec u, U)=\imT(\vec b, U)\subset \widetilde {\Omega}_{\vec b},
$$
thus proving the $L^\infty$ bound. In addition,
$$
	\mc{N}(\vec u,\O_d \cap U, \cdot)
	\overset{\text{a.e.}}{=} \deg (\vec u, U, \cdot)
	= \deg (\vec b, U, \cdot ).
$$
As
$\vec b$ is an orientation-preserving diffeomorphism,
\[
 \deg (\vec b, U, \cdot) = \begin{cases}
 1 & \text{in } \vec b (U) , \\
 0 & \text{in } \R^3 \setminus \vec b (\overline{U}) , \\
 \text{undefined} & \text{in } \vec b (\p U)
 \end{cases}
\]
and $\vec b (\p U)$ has measure zero. We conclude that $\mc{N}(\vec u,\O_d \cap U, \cdot)
	=\deg (\vec b, U, \cdot) = \chi_{\vec b (U)}$ a.e.,
which implies that 
$\vec u$ is injective a.e.\ in $U$.
As this is true for all $U \in \mc{U}_{\vec u}$ with $\O \Subset U$, and $\widetilde{\O}$ can be written as the union of countably 
many such $U$, we conclude that $\vec u$ is injective a.e.\ and $\imG (\vec u, \widetilde{\O}) = \vec b (\widetilde{\O})$ a.e.
Since the $L^\infty$ bound has already been established, the identity $\Det D\vec u=\det D\vec u$ can be proved exactly as in \cite[Th.\ 4.1]{BaHeMo17}.

Take now an arbitrary $U\in \mathcal U_{\vec u}$ (on whose boundary $\vec u$ is not necessarily continuous).
Proceeding as in \cite[Thm.~4.1]{BaHeMo17} one still obtains that there exists $c\in \Z$ such that 
$$\mc{N}(\vec u, \O_d\cap U, \cdot )-\deg(\vec u, U, \cdot)=c
\quad\text{a.e.},$$
where $\deg (\vec u, U, \cdot)$ is now the Brezis--Nirenberg degree
(see Definition \ref{prop:degree} and the remark after it).
Since $\vec u$ is injective a.e.\@, the first term coincides a.e.\@ (see Lemma \ref{le:Lusin})
with $\chi_{\imG(\vec u, U)}$.
As $\vec u \in L^{\infty} (\O, \R^3)$, there exists a set $N \subset \O$ of measure zero such that $\vec u (\O \setminus N) \subset B (\vec 0, \| \vec u \|_{L^{\infty}})$.
Therefore, $\imG (\vec u, \O) \subset B (\vec 0, \| \vec u \|_{L^{\infty}}) \cup \vec u (\O_0 \cap N)$.
The set $\vec u (\O_0 \cap N)$ has measure zero thanks to Lemma \ref{le:Lusin}.
Thus, $\chi_{\imG(\vec u, U)}= 0$ a.e.\ outside $\R^3 \setminus B (\vec 0, \| \vec u \|_{L^{\infty}})$.
By Definition \ref{prop:degree}, $\deg (\vec u, \p U, \cdot)= 0$ a.e.\ outside $B (\vec 0, \| \vec u \|_{L^{\infty}})$.
Consequently, $c = 0$ and \eqref{eq:imG=imT} holds.
\end{proof}

Next, we show that the inverse of a map $\vec u$ in $\mathcal{A}_s^r$ has $W^{1,1}$ regularity. 
This follows essentially from \cite[Th.\ 3.4]{HeMo15};
nevertheless, some clarifying statements are in order due to the potential lack of continuity of \(\vec u\) on \(L\). 

\begin{proposition}\label{INV2}
Let $\vec u \in \Asr$. Then $\vec u^{-1} \in W^{1,1} (\widetilde \Omega_{\vec b}, \R^3)$.
\end{proposition}

\begin{proof}
Following \cite{CoDeLe03}, in the $H^1$ setting the sets 
$$A_{\vec u,U}:= \{\vec y\in \R^3: \deg( \vec u, U ,\vec y)\neq 0 \}$$
are given an auxiliary role, 
the actual topological images being now defined as
$$
	\imTBN(\vec u, U):=\{\vec y\in \R^3:
	D\big ( A_{\vec u, U}, \vec y\big ) =1 \}.
$$
The superscript `BN' has been added to indicate that use is made of the Brezis--Nirenberg degree. 
That notation and definition will only appear in this proof. 

\medskip

\underline{Part I: if $\vec u\in \Asr$ then for any $U\in \mathcal{U}_{\vec u}$,}
\begin{equation} \label{eq:imG=imT-BN}
	\imG (\vec u, U) = \imTBN (\vec u, U)\ \text{a.e.} 
\end{equation}

The proof consists in recalling \eqref{eq:imG=imT}, which implies that 
\begin{equation}\label{eq:imGAuU}
 \imG (\vec u, U) = A_{\vec u,U} \quad \text{a.e.}
\end{equation}
Now, by Lebesgue's differentiation theorem,
\[
 \imG (\vec u, U) = \{ \vec y \in \R^3 : D (\imG (\vec u, U), \vec y) = 1 \} \quad \text{a.e.}
\]
Using \eqref{eq:imGAuU} as well, we conclude \eqref{eq:imG=imT-BN}.
\medskip

\underline{Part II: maps in $\mathcal{A}_s^r$ satsify INV in the whole $\widetilde \Omega$.}
Let $U \in \mathcal{U}_{\vec u}$ and assume that $\vec u|_{U\setminus N}$ is injective for some set $N \subset U$ of measure zero. 
Take $\vec a \in U$ and define $r_{\vec a} := \dist (\vec a, \p U)$.
Then $B (\vec a, r) \in \mc{U}_{\vec u|_{U}}$ for a.e.\ $r \in (0, r_{\vec a})$.
Fix any such $r$.
For all $\vec x \in B (\vec a, r) \cap \O_0$ we have that $\vec u (\vec x) \in \imG (\vec u, B (\vec a, r))$.
By \eqref{eq:imG=imT-BN} and Lemma \ref{le:Lusin}, we infer that $\vec u (\vec x) \in \imTBN (\vec u, B (\vec a, r))$ for a.e.\ $\vec x \in B (\vec a, r)$.
Now, for all $\vec x \in U \setminus (B(\vec a, r) \cup N)$, by the injectivity, $\vec u (\vec x) \notin \imG (\vec u, B(\vec a, r))$.
As before, $\vec u (\vec x) \notin \imTBN (\vec u, B(\vec a, r))$ for a.e.\ $\vec x \in U \setminus B(\vec a, r)$.
We have then shown that $\vec u|_{U}$ satisfies the condition INV for $H^1$ maps
(with the topological images defined using the Brezis--Nirenberg degree).
As $\widetilde \O$ can be written as the union of countably many $U \in \mc{U}_{\vec u}$, we conclude that $\vec u$ satisfies condition INV.

\medskip

\underline{Part III: regularity of the inverse when $\vec u\in \Asr$.}
Once condition INV for $H^1$ maps 
has been established,
we obtain by \cite[Th.\ 3.4]{HeMo15}
that  the extension of $\vec u^{-1}$ by zero to all $\R^3$ is in $SBV$ and the restriction of $D \vec u^{-1}$ to $\imTBN (\vec u, U)$
is absolutely continuous with respect to the Lebesgue measure
for any $U\in \mathcal{U}_{\vec u}$. 
Apply  this to any $U\in \mathcal{U}_{\vec u}$
such that $\Omega\Subset U \Subset \widetilde \Omega$.
Since $\vec b$ is a $C^1$ orientation-preserving diffeomorphism, 
applying the measure-theoretic inverse function theorem in \cite[Lemma 2.5]{MuSp95}
we find that $\imTBN(\vec b, U) = \imT(\vec b, U)$.
As $\imT (\vec b, U)$ is open, we obtain that $\vec u^{-1} \in W^{1,1} (\imT (\vec b, U), \R^3)$.
Since $\vec u^{-1}=\vec b^{-1}$ in $\widetilde \Omega_{\vec{b}}\setminus \imT(\vec b, U)$, 
it follows that $\vec u^{-1} \in W^{1,1} (\widetilde \Omega_{\vec b}, \R^3)$,
as desired.
\end{proof}

\subsection{Pointwise convergence of inverses}

The following result shows that the inverse is stable under the weak limit in $H^1$.

\begin{lemma}\label{le:convu-1}
For each $j \in \N$, let $\vec u_j , \vec u \in H^1 (\widetilde\Om, \R^3)$ be axisymmetric.
Assume that $\vec u_j \weakc \vec u$  in $H^1(\widetilde\Om, \R^3)$ as $j \to \infty$.
Suppose that $\det D \vec u_j >0$ a.e.\ for all $j \in \N$ and $\det D \vec u >0$ a.e., and that $\vec u_j$ and $\vec u$ are invertible a.e\@.
Then $\vec u_j^{-1} \to \vec u^{-1}$ a.e.
\end{lemma}
\begin{proof}
By Lemma \ref{le:weakconv}, $\vec u$ is axisymmetric.
Let $\vec v_j$ and $\vec v$ be the corresponding \(2D\) functions to $\vec u_j$ and $\vec u$, respectively.
By Lemmas \ref{le:uvSobolev}, \ref{le:uvinj} and \ref{N condition}, $\det D \vec v_j > 0$ a.e., $\vec v_j$ is injective a.e., 
$\vec v_j \in H^1 (\pi (\widetilde\O) \setminus ([0, \d] \times \R)))$ for any $\d > 0$, and analogously for $\vec v$.
Moreover, $\vec v_j \weakc \vec v$ in $H^1 (\pi (\widetilde\O) \setminus ([0, \d] \times \R)))$ for each $\d>0$, and 
$\E (\vec v_j, \pi (\widetilde\O) \setminus (\{0\} \times \R)) = 0$, and analogously for $\vec v$.
By \cite[Th.\ 6.3]{BaHeMo17}, $\vec v^{-1}_j \to \vec v^{-1}$ a.e\@.
Arguing as in Lemma \ref{le:weakconv} (on the inverses) shows that $\vec u^{-1}_j \to \vec u^{-1}$ a.e.
\end{proof}

\subsection{The horizontal components of the inverse have no singular parts on $\imT(\vec u, L)$}

For general maps in $\As$ the equality $\widetilde\Om_{\vec b}=\imG(\vec u,\widetilde\Om)$ does not hold in general, 
as can be seen by the classical example of a radial cavitation.
Furthermore, even when $\imG(\vec u,\widetilde\Om)$ does coincide a.e.\@ with $\widetilde\Om_{\vec b}$,
the inverse is not necessarily in $W^{1,1}(\widetilde \Omega_{\vec b}, \R^3)$ as shown by the example of 
Conti \& De Lellis \cite{CoDeLe03} where $\imT(\vec u, L)$ consists (apart from the symmetry axis) of the sphere 
$\partial B\big ( (0,0,\frac{1}{2}), \frac{1}{2} \big )$ in the deformed configuration
and $\vec u^{-1}$ has a jump across this sphere.
Nevertheless, in the following lemma we show that any such singularities in $\vec u^{-1}$
are due to the vertical component $u_3^{-1}$ of the inverse, whereas its horizontal components 
$u_1^{-1}$ and $u_2^{-1}$ enjoy a Sobolev regularity.

From now on, for $\alpha \in \{1, 2, 3\}$, we denote by $u^{-1}_\alpha$ the $\alpha$-th component 
of $\vec u^{-1}$. Recall that equalities \eqref{eq:Om'} hold and that
\begin{align}
		\label{le:imL_union_imOminusL}
	 \widetilde{\O}_{\vec b}= \imT(\vec u, \widetilde\O \setminus L) \cup \imT(\vec u, L)
	 \end{align}
	  by Definition \ref{def:imT(u,L)} and Lemma \ref{le:contained_in_target_domain}.(b).
	We shall need the following gluing theorem for $BV$ functions \cite[Th.~3.84]{AmFuPa00}.
	
	\begin{proposition}
			\label{le:BVglue}
		Let \( N,m \geq 1\). Let \(\O \subset \R^N\) be an open set. Let $\vec u, \vec v \in BV(\Omega, \R^m)$ and 
		let $E$ be a set of finite perimeter in $\Omega$, with $\partial^* E \cap \Omega$
		oriented by $\vecg\nu_E$.
		Let $\vec u^+_{\partial^*E}$, $\vec v^-_{\partial^*E}$ be the traces of $\vec u$ and $\vec v$ on $\partial^* E$, which are 
		defined for $\mathcal H^2$-a.e.\ point of $\partial^*E$.
		Set $\vec w= \vec u\chi_E + \vec v\chi_{\Omega\setminus E}$.
		Then
		$$
			\vec w \in BV(\Omega, \R^m)
			\quad \Leftrightarrow\quad
			\int_{\partial^*E \cap \Omega}
			|\vec u^+_{\partial^*E} - \vec v^-_{\partial^*E}| \, \dd\mathcal H^{N-1} < \infty ,
		$$
and in this case,
		$$
			D \vec w = D \vec u\res E^1 + (\vec u^+_{\partial^*E}- \vec v^-_{\partial^*E})\otimes \vecg\nu_{E} \mathcal H^{N-1}\res (\partial^*E\cap \Omega) + D \vec v\res E^0,
		$$
		where $E^0$ and $E^1$, respectively, denote the set of points at which
		$E$ has density $0$ and $1$.
	\end{proposition}

\begin{proposition}\label{regularity first 2 components}
Let $\vec u \in \As$ and \(\alpha=1,2 \). Denote by $\widehat{u^{-1}_\alpha} : \widetilde{\O}_{\vec b} \to \R$ the map
\[
 \widehat{u^{-1}_\alpha} = \begin{cases}
 u^{-1}_\alpha & \text{in } \imT(\vec u, \widetilde\O \setminus L)=\widetilde{\O}_b \setminus \imT(\vec u,L) , \\
 0 & \text{in } \imT(\vec u, L) .
 \end{cases}
\]
Then, $\widehat{u^{-1}_\alpha}\in W^{1,1} (\widetilde\Om_{\vec b})$ and 
$\widehat{u^{-1}_\alpha}$ has a precise representative whose restriction to the complement
of a certain $\mathcal{H}^2$-null set is continuous.
\end{proposition}
\begin{proof}
Let $\vec v$ be the planar function corresponding to $\vec u$.
As usual we identify $\vec v$ with its continuous representative in 
$\pi(\widetilde \Omega )\setminus(\{0\}\times \R)$.
\medskip

\underline{Part I: covering the $2D$ domain $\pi(\Omega)\setminus (\{0\}\times \R)$
	with an increasing sequence of good}
	\\
	\underline{open sets, ever closer to the singular segment.}
	
Thanks to Lemma \ref{le:uvSobolev}, $\vec v \in H^1 (\pi(\widetilde\Om) \setminus ([0, \d] \times \R ), \R^2)$ for each $\d>0$.
Then, by Lemma \ref{le:enough_good_open_sets}, for a.e.\ small $\d>0$ it is possible to find\footnote{
In principle, we would like to use $E_\delta:= \pi(\widetilde\Omega)\setminus ([0,\delta]\times \R)$
itself. However, our analysis is based on previous works where having a $C^2$ boundary (and not just piecewise $C^2$), 
as well as being compactly contained in the working domain, were added as requirements for membership to the class of good open sets
(for the sake of achieving concise statements such as the assertion `$U_t\in \mathcal{U}_{\vec u}$ for a.e.\ $t$' in Lemma \ref{le:enough_good_open_sets}).
}
an open set $E^{(\delta)}\Subset \pi(\widetilde \Omega )\setminus (\{0\}\times \R)$, with a \(C^2\) boundary, such that 
\begin{equation*}
		\partial E^{(\delta)} \cap \pi(\Omega)
		=\{(r,x_3)\in \pi(\Omega): r=\delta\}
\end{equation*}
and $E^{(\delta)} \in \mathcal{U}_{\vec v}$ (that is, all the analogous properties of Definition \ref{def:good_open_sets} are satisfied for the 
planar map $\vec v$). Indeed, it can be seen that for every small $c>0$ it is possible to construct a $C^2$ open set 
$E\Subset \pi(\widetilde\Omega\setminus (\{0\}\times \R)$ such that 
$$
\partial E \cap \pi(\Omega) = \{(r, x_3)\in \pi(\Omega): r=c\}
$$
(begin with the set $(\{c\}\times \R)\cap \pi(\Omega)$, which consists of a finite number of segments and is nonempty because $c$ is small;
stretch this set vertically so as to obtain a new finite union of segments containing the former ones but having their endpoints on
$\pi(\Omega_D)=\pi(\widetilde\Omega\setminus \overline \Omega)$; then close the loop---or loops---with a $C^2$ curve entirely contained in $\pi(\Omega_D)$).
Applying Lemma \ref{le:enough_good_open_sets} to $E$ we find that $E_t\in \mathcal{U}_{\vec v}$ for a.e.\ $t>0$, with $E_t$ defined in \eqref{eq:def_distance_level_sets}.
It can be seen that 
$$
\partial E_t \cap \pi(\Omega) = \{(r, x_3)\in \pi(\Omega): r=c-t\}.
$$
Recall that $E_t\in \mathcal{U}_{\vec v}$ for every small $c>0$ and a.e.\@ small $t>0$.
Since a.e.\@ small $\delta>0$ can be written as $\delta=c-t$ for an appropriate choice of $c$ and $t$, the claim follows.

\medskip

\underline{Part II: the inverse and the definition of the $\mathcal{H}^1$-null exceptional set.}
	
From the comments in Subsection \ref{sob reg inverse}, 
$\vec v^{-1}\in W^{1,1}(\imT (\vec v, E^{(\delta)}),\R^2)$ and it is continuous
at every point of $\imT (\vec v, E^{(\delta)})\setminus T_\delta$, where
$T_\delta$ has zero $\mathcal{H}^1$-measure and it is defined by
\begin{multline*}
T_\delta:=\Big\{(s,y_3)\in \imT(\vec v, E^{(\delta)}): \text{ there exist at least}
\\
\text{two different points }
(r,x_3)\in \overline{E^{(\delta)}}
\text{ such that }
(s,y_3)\in \imT\big (\vec v, (r,x_3)\big ) \Big \}.
\end{multline*}
Observe that for every $(r,x_3)\in \pi(\widetilde\Omega)\setminus (\{0\}\times \R)$,
\begin{align}\label{eq:left_inverse}
\vec v(r,x_3)\in \imT(\vec v, E^{(\delta)})\setminus T_\delta
\quad \Rightarrow\quad 
\vec v^{-1}\big ( \vec v(r,x_3) \big )= (r,x_3).
\end{align}

Fix a sequence $(E_k)_{k\in \N}$ of the sets $E^{(\delta)}$ of Part I
with $E_k=E^{(\delta_k)}$ and $(\delta_k)_{k\in \N}$ a sequence decreasing to zero. 
Without loss of generality, assume that $E_{k}\Subset E_{k+1}$ for all $k$,
that each $E_k$ is the projection by $\pi$ of a good $3D$ open set $V_k \in \mathcal{U}_{\vec u}$,
and that $\pi(\widetilde\Omega) \setminus (\{0\}\times\R) = \bigcup_{k \in \mathbb{N}} E_k$.
Set
$$
T:= \bigcup_{k\in \N} T_{\delta_k} \cup
\bigcup_{k\in\N} \vec v(\partial E_k\setminus \pi(\Omega_0)),
$$
where $\Omega_0$ is the set of Definition \ref{def:Om0}.
By property (b) in Definition \ref{def:good_open_sets} and the fact that $H^1(\Gamma,\R^2)$
functions on a smooth curve $\Gamma$ map $\mathcal{H}^1$-null sets onto $\mathcal H^1$-null sets
(see \cite{MaMi73}, \cite[Prop.~2.7]{MuSp95}), it follows that $\mathcal H^1(T)=0$.

Denote by $w : \pi(\widetilde \Om_{\vec b}) \to \R$ the function
\[
 w = \begin{cases}
 v^{-1}_1 & \text{in } \pi (\imT(\vec u, \widetilde \O \setminus L)) , \\
 0 & \text{in } \pi (\imT(\vec u, L)) ,
 \end{cases}
\]
where \(v_1^{-1}\) is the first component of the inverse of \(\vec v\). 
In particular \(\vec u^{-1}= v^{-1}_1(\cos \theta \vec e_1+\sin \theta \vec e_2)+v^{-1}_2\vec e_3\).

From \eqref{eq:bdryimT}  we have $\p \imT(\vec v,E_k) \subset \vec v(\partial E_k)$ for all $k$.
	But \( \p E_k
	=\big (\p E_k \cap \pi(\widetilde \O\setminus \Omega)\big )  \cup \big( (\{\d_k\}\times \R) \cap \pi(\O)  \big)\). Thus we find
	\( \p \imT(\vec v,E_k) \subset \vec v\big( \pi(\widetilde \O\setminus \Omega)\big )
	\cup \vec v \left( (\{\d_k\}\times \R) \cap \pi(\O)\right)\). However, thanks to the boundary condition \(\vec b\), we know that \(\vec v\big( \pi(\widetilde \O\setminus \Omega)\big )=\pi \big(\vec b(\widetilde\Omega \setminus \Omega)\big)\).
	 Hence \( \pi(\O_{\vec b})\cap \vec v\big( \pi(\widetilde \O\setminus \Omega)\big )=\varnothing\). But we also have \( \vec v(\pi(\O)) =\pi(\O_{\vec b})\). Thus we  infer that \( \vec v \big( (\{\d_k\}\times \R)\cap \pi(\O) \big)\cap \pi(\O_{\vec b})=\vec v\big( (\{\d_k\}\times \R) \cap \pi(\O) \big)\) and we deduce	
	\begin{align}
		\label{eq:bdryE_k}
		\p \imT(\vec v, E_k)\cap \pi(\O_\vec b)= \vec v\big((\{\d_k\}\times \R)\cap \pi(\O) \big).
	\end{align}
	
Let $(\delta_k, x_3)\in \pi(\Omega)$ be such that $(s,y_3):=\vec v (\delta_k,x_3) \in \pi(\Omega_{\vec b})\setminus T$. 
We claim that 
\begin{align}
			\label{eq:inverse_circ_v}
		(s,y_3)\in 
		\pi\big (\imT(\vec u, \widetilde\Omega\setminus L)\big ),
			\quad
		w \text{ is continuous at } (s,y_3),
			\quad \text{and} \quad
		w(s,y_3)=\delta_k. 
\end{align}
First, we show that $(s,y_3)\in \imT(\vec v, E_{k+1})$. Suppose, for a contradiction, that this were false. 
Since, by continuity, it is easy to see that $(s,y_3)\in \overline{\imT(\vec v, E_{k+1})}$, by \eqref{eq:bdryimT}
it follows that $(s,y_3)= \vec v(\delta_{k+1}, x_3')$ for some $(\delta_{k+1}, x_3')\in \pi(\Omega)$. 
But $(s, y_3)\notin T$, so both $(\delta_k,x_3)$ and $(\delta_{k+1}, x_3')$
belong to $\Omega_0$. Thus, on the one hand, $\vec v(\delta_k, x_3)=\vec v(\delta_{k+1}, x_3')$. 
On the other hand, $\vec u$ is injective in $\Omega_0$ \cite[Lemma 3]{HeMo11}, yielding a contradiction.
	
By Lemma \ref{le:imTuimTv} and Definition \ref{def:imT(u,L)}, we conclude that 
$(s,y_3)$ indeed belongs to $\imT(\vec u, \widetilde \Omega\setminus L)$. 	
In addition, since $(s,y_3)\in \imT(\vec v, E_{k+1})\setminus T_{k+1}$, 
$\vec v^{-1}$ is continuous at $(s,y_3)$.
By \eqref{eq:left_inverse}, $\vec v^{-1}(s,y_3)=(\delta_k, x_3)$ and
$w(s,y_3)=\delta_k$.

	Combining \eqref{eq:bdryE_k} with \eqref{eq:inverse_circ_v} we find that
	if $(s,y_3)\in \partial \imT(\vec v, E_k) \cap \pi (\Omega_{\vec b}) \setminus T$
	and $(s^{(j)}, y_3^{(j)})$
	is any sequence converging to $(s,y_3)$ 
	then $w\big ( (s^{(j)}, y_3^{(j)}) \big )
	\overset{j\to\infty}{\longrightarrow}
	\delta_k$.
	This is the motivation for the definition
	of the exceptional set $T$.

	\medskip
	\underline{Part III: $\imT(\vec v, E_k) \cup 
	 \pi \Big ( \vec b \big (\widetilde\Omega\setminus (\Omega \cup \overline C_{\delta_k})\big ) \Big )
	 	= \imT\big ( \vec v, 
 		\pi(\widetilde\Omega)\setminus
 		([0,\delta_k]\times\R)\big)$
 		 for every $k\in \N$.}
 	In preparation for using the excision 
 	property of the degree,
 	let us show first that 
 	\begin{align}
	 		\label{eq:pre-excision1}
 		\text{if }
 		\vec y'_0 \in \imT(\vec v, E_k)
 			\quad 
 		\text{then}
 			\quad 
 		\vec y'_0 \notin 
 		\vec v \Big ( 
 			\partial E_k
 			\cup
 			\partial \big (
 				\pi(\widetilde\Omega)\setminus
 		\big([0,\delta_k]\times\R\big)
 		\big ) \Big ).
 	\end{align}
 	By definition of topological image, 
 	$\vec y_0'\notin \vec v(\partial E_k)$.
 	Suppose, for a contradiction, that
 	$\vec y_0' \in \vec v \Big ( 
 	\partial \big (
 				\pi(\widetilde\Omega)\setminus
 		\big([0,\delta_k]\times\R\big)
 		\big ) 
 		\setminus \partial E_k \Big )$.
 	Then $\vec y_0'=\vec v(\vec x_0')$
 	for some $\vec x_0'=(r_0, x_3) \in 
 	\partial \pi \big( \widetilde \Omega \big)$
 	with $r_0\geq \delta_k$,
 	or some $\vec x_0'=(r_0, x_3)
 	\in \partial \Big ( \pi(\widetilde\Omega)\setminus\big (\overline{E_k}\cup ([0,\delta_k]
 		\times \R)\big )\Big )$
 	with $r_0=\delta_k$.
 	In both cases,
 	$\vec x_0'$ is in the region where $\vec v$
 	is defined through the boundary data $\vec b$.
 	By \cite[Lemma 2.5]{MuSp95}
 	$$
 		D\left( 
 				\vec v \Big ( \pi(\widetilde \Omega) \setminus\big (\overline{E_k}\cup ([0,\delta_k]
 		\times \R)\big )		
 		\Big),
 		\vec y_0' \right) \geq \frac{1}{2}.
 	$$
 	At the same time, since $\vec y_0'$
 	belongs to the open set $\imT(\vec v, E_k) \overset{\text{a.e.}}{=} \imG(\vec v,E_k)$,
 	using \eqref{eq:deg_positive} and 
 	Lemma \ref{le:imTuimTv}
 	we find that
 	$$
 		\mathcal L^2 \left(\imG(\vec v, E_k) \cap 
 		\vec v \Big ( \pi(\widetilde \Omega) \setminus\big (\overline{E_k}\cup ([0,\delta_k]
 		\times \R)\big )		
 		\Big )\right) >0.
 	$$
 	Since $\vec v$ is injective a.e.\@ 
 	we arrive at a contradiction.
 	
 	A similar proof yields that 
 	\begin{multline*}
 		\text{if }
 		\vec y'_0 \in \imT\Big ( \vec v, 
 		\pi(\widetilde\Omega)\setminus
 		\big([0,\delta_k]\times\R\big)\Big )
 		\setminus 
 		\pi \Big ( \vec b \big (\widetilde\Omega\setminus (\Omega \cup \overline C_{\delta_k})\big ) \Big )
 		\\
 		\text{then}
 			\quad 
 		\vec y'_0 \notin 
 		\vec v \Big ( 
 			\partial E_k
 			\cup
 			\partial \big (
 				\pi(\widetilde\Omega)\setminus
 		\big([0,\delta_k]\times\R\big)
 		\big ) \Big ).
 	\end{multline*}
 	Having both relations, let us now prove 
 	that 
 	$
 		\imT(\vec v, E_k) \subset 
 		\imT\big ( \vec v, 
 		\pi(\widetilde\Omega)\setminus
 		\big([0,\delta_k]\times\R\big)
 		\big ).
 	$
 	Let $\vec y_0' \in \imT(\vec v,E_k)$. By \eqref{eq:pre-excision1},
 	the excision property gives that
	 	\begin{equation}
 			\label{eq:excision311}
 		\deg(\vec v, E_k, \vec y_0')
 		+ \deg (\vec v, 
 		\pi(\widetilde \Omega) \setminus\big (\overline{E_k}\cup ([0,\delta_k]
 		\times \R)\big ),
 		\vec y_0')
 		=
 		\deg (\vec v, \pi(\widetilde\Omega)\setminus
 		\big([0,\delta_k]\times\R\big),
 		\vec y_0').
 	\end{equation}

 	Since the restriction of $\vec v$ to 
 	$\pi(\widetilde \Omega) \setminus\big (\overline{E_k}\cup ([0,\delta_k]
 		\times \R)\big )$
 	is an 
 	orientation-preserving diffeomorphism
 	(given by the planar function corresponding
 	to the axisymmetric boundary data $\vec b$),
 	then 
 	a proof similar to that 
 		of \eqref{eq:pre-excision1} shows that
 	$\vec y_0'\notin \vec v \Big (\pi(\widetilde \Omega) \setminus\big (\overline{E_k}\cup ([0,\delta_k]
 		\times \R)\big )\Big )$.
 	Hence, the second degree in 
 	\eqref{eq:excision311}
 	is zero. 
 	By definition of $\imT(\vec v, E_k)$, 
 	the first degree in \eqref{eq:excision311} is 1,
 	hence $\vec y_0'\in \imT\big (\vec v, \pi(\widetilde\Omega)\setminus
 		\big([0,\delta_k]\times\R\big)
 		\big )$, as desired.
 		
 	The proof that 
 	$
 		\pi \Big ( \vec b \big (\widetilde\Omega\setminus (\Omega \cup \overline C_{\delta_k})\big ) \Big ) \subset 
 		\imT\big ( \vec v, 
 		\pi(\widetilde\Omega)\setminus
 		\big([0,\delta_k]\times\R\big)
 		\big )
 	$
 	is easier, because in that region $\vec v$
 	is dictated by the diffeomorphism $\vec b$. 
 	Finally,
 	to prove that 
 	$$\imT\Big ( \vec v, 
 		\pi(\widetilde\Omega)\setminus
 		\big([0,\delta_k]\times\R\big)\Big )
 		\setminus 
 		\pi \Big ( \vec b \big (\widetilde\Omega\setminus (\Omega \cup \overline C_{\delta_k})\big ) \Big )
 		\subset \imT(\vec v, E_k),$$
 	suppose that $\vec y_0'$ belongs
 	to the set on the left. 
 	Then, by \eqref{eq:excision311},
 	$$
	 	\deg(\vec v, E_k, \vec y_0')
 		+ \deg (\vec v, 
 		\pi(\widetilde \Omega) \setminus\big (\overline{E_k}\cup ([0,\delta_k]
 		\times \R)\big ),
 		\vec y_0')
 		\ne 0.
 	$$
 	The second term is zero
 	since
 	$\vec v$ restricted to 
 	$\pi(\widetilde \Omega) \setminus\big (\overline{E_k}\cup ([0,\delta_k]
 		\times \R)\big )$
 	is a diffeomorphism
 	and 
 	$\vec y_0'\notin \pi \Big ( \vec b \big (\widetilde\Omega\setminus (\Omega \cup \overline C_{\delta_k})\big ) \Big )$,
 	a set which contains 
 	$\vec v\Big ( \pi(\widetilde \Omega) \setminus\big (\overline{E_k}\cup ([0,\delta_k]
 		\times \R)\big )\Big )$.
 	Consequently, the first degree is nonzero and
 	$\vec y_0'\in \imT(\vec v, E_k)$,
 	finishing the proof of this Part III.

\medskip
\underline{Part IV: $\mathcal{H}^2$-continuity of $\widehat{u_\alpha^{-1}}$.}
For each $k\in\N$
define $w_k: \pi(\widetilde \Om_{\vec b}) \to \R$ as
\[
 w_{k} = \begin{cases}
 v^{-1}_1 & \text{in } 
 	\imT\big ( \vec v, 
 		\pi(\widetilde\Omega)\setminus
 		\big([0,\delta_k]\times\R\big), \\
 \d_k & \text{otherwise,}
 \end{cases}
\]
where $v_1^{-1}$ is defined in terms of $\vec b^{-1}$ in 
$\pi \Big ( \vec b \big (\widetilde\Omega\setminus (\Omega \cup \overline C_{\delta_k})\big ) \Big )$.
By Parts II and III, the function $w_k$ is  continuous at every point in
$\pi(\widetilde\Omega_{\vec b})\setminus T$. 
Since
\begin{equation*}
\sup_{\pi(\widetilde\Omega_{\vec b})\setminus T} |w-w_k|=\delta_k,
\end{equation*}
we have that $w_k \to w$ uniformly 
 in
$\pi(\widetilde\Omega_{\vec b})\setminus T$
as $k \to \infty$.

For the bound $|w-w_k|\leq \delta_k$ 
	in the image of $\pi(\widetilde \Omega) \cap \big ((0,\delta_k]\times \R\big )$
	we use that 
	if $(r,x_3)\in \pi(\Omega)$ with $r>0$ and $\vec v(r,x_3) \in \pi(\Omega_{\vec b})\setminus T$ then $\vec v(r,x_3)\in \pi\big ( \imT(\vec u, \widetilde \Omega\setminus L)\big )$, 
	$w$ is continuous at $\vec v(r,x_3)$, and
	$w\big (\vec v(r,x_3)\big ) = r$.
	That can be proved similarly as Part II,
	finding $k$ such that $\delta_k < r$
	and
	assuming that $(s,y_3)=\vec v(r,x_3)$ is both on
	$\vec v(\partial E_k)$ and on 
	$\vec v(\partial E_{k+1})$.

Therefore, $w|_{\pi(\widetilde\Omega_{\vec b})\setminus T}$
 is continuous.
Since \(\widehat{u^{-1}_1} \vec e_1+ \widehat{u^{-1}_2}\vec e_2=w(\cos \theta \vec e_1+\sin \theta \vec e_2)\) we have then that $\widehat{u^{-1}_\alpha}$ 
has a precise representative 
whose restriction to the complement of a certain set of zero $\mathcal{H}^2$-measure
(the preimage by $\pi$ of $T$)
is continuous.

	\underline{Part V: Sobolev regularity of $\widehat{u_\alpha^{-1}}$.}
	Let $V_k\Subset \widetilde \Omega$ be the good (3D) open set such that
	$\pi(V_k)=E_k$.
	By \cite[Prop.~2.17.(vi)]{HeMo12}, 
	$\imT(\vec u, V_k)$
	has finite perimeter and $\partial^*\imT(\vec u, V_k) = \imG(\vec u, \partial V_k)$ $\mathcal{H}^2$-a.e. 
	The set $\vec b\big ( \widetilde \Omega \setminus (\Omega \cup \overline C_{\delta_k})\big )$
	also has finite perimeter
	since $\vec b$ is diffeomorphism up to the boundary of $\widetilde \Omega$. 
	By Part III and Lemma \ref{le:imTuimTv},
	\begin{align}
			\label{eq:imTCdelta}
		\imT(\vec u, V_k) \cup \vec b
		\big ( \widetilde \Omega \setminus (\Omega \cup \overline C_{\delta_k})\big )
		= \imT(\vec u, \widetilde\Omega
		\setminus \overline C_{\delta_k})
	\end{align}
	for every $k\in \N$. 
	Hence, 
	$\imT(\vec u, \widetilde\Omega
		\setminus \overline C_{\delta_k})$
	is a set of finite perimeter. 
	By Lemma \ref{le:u-1} the map 
	$\vec u^{-1}_{V_k}: \R^3 \to \R^3$
	given by
	$$
		\vec u^{-1}_{V_k}(\vec y)=
		\begin{cases}
			\vec u^{-1}(\vec y),
			& \vec y \in \imT(\vec u, V_k), \\
			\vec 0, & \vec y \in \R^3 \setminus 
			\imT(\vec u, V_k)
		\end{cases}
	$$
	is in $SBV(\R^3, \R^3)$
	and 
	$D\vec u^{-1}_{V_k} \res \imT(\vec u, V_k)$ is absolutely continuous. 	
	Applying Proposition \ref{le:BVglue}
	to $\vec u_{V_k}^{-1}$ and 
	to $\vec b^{-1}$,
	with $\imT(\vec u, V_k)$
		as the set of finite perimeter
		in the hypotheses of that gluing theorem,
	we obtain that the map
	$$
		\vec y\in \widetilde \Omega_{\vec b}
		\mapsto 
		\begin{cases}
			\vec u^{-1}(\vec y),
			& 
			\vec y \in \imT(\vec u, V_k) \\
		\vec b^{-1}(\vec y),
			& \vec y \in 
			\widetilde \Omega_{\vec b} \setminus  \imT(\vec u, V_k) 
		\end{cases}
	$$
	is in $SBV(\widetilde\Omega_{\vec b}, \R^3)$,
	with derivative given by 
	\begin{multline*}
		D\vec u^{-1}_{V_k}\res \imT(\vec u, V_k)
		+ D\vec b^{-1}\res \big ( \widetilde\Omega_{\vec b} \setminus \overline{\imT(\vec u, V_k)}\big )
		\\
		+ \big ( (\vec u^{-1})^+
			- (\vec b^{-1})^-\big )
			 \vecg\nu_{\imT(\vec u, V_k)} 
			 \mathcal H^2\res\imG(\vec u, \partial V_k).
	\end{multline*}
	(The set $\imT(\vec u, V_k)$ has neither density zero nor one
	at $\mathcal H^2$-a.e.\ point in
	$\partial \imT(\vec u, V_k)$
	thanks to \cite[Lemma 2.5]{MuSp95})
	Since $\vec u=\vec b$ in $\Omega_D$,
	taking \eqref{eq:imTCdelta}
	into account,
	the map can be rewritten as
	$$
		\vec y\in \widetilde \Omega_{\vec b}
		\mapsto 
		\begin{cases}
			\vec u^{-1}(\vec y),
			& 
			\vec y \in 
			\imT(\vec u, \widetilde\Omega
		\setminus \overline C_{\delta_k}) \\
		\vec b^{-1}(\vec y),
			& \text{otherwise}, 
		\end{cases}
	$$
	with a corresponding rewriting for the derivative.
	At this point,
	taking into account 
	\eqref{eq:imTCdelta},
	we apply Proposition \ref{le:BVglue}
	again, now to the 
	first two components of the above map
	and to
	the function 
	$$
		\vec y=(s\cos\theta, s\sin\theta, y_3)
		\in \widetilde\Omega_{\vec b}
		\mapsto 
		\delta_k (\cos\theta, \sin \theta)
	$$
	(which belongs to $W^{1,1}(\widetilde \Omega_{\vec b}, \R^2)$),
	with 
	$\imT(\vec u, \widetilde\Omega
		\setminus \overline C_{\delta_k})$
 	as the set of finite perimeter in the
 	hypothesis of that gluing theorem,
	to find that
	the map
	$\vec W_k:\widetilde\Omega_{\vec b}\to \R^2$
	given by
	\begin{align}
			\label{eq:w_k_3D}
		\vec W_k:\vec y=(s\cos\theta, s\sin\theta, y_3)
		\in \widetilde\Omega_{\vec b}
		\mapsto 
		\begin{cases}
			\big ( u_1^{-1}(\vec y),
			u_2^{-1}(\vec y) \big ) 
			& \vec y  \in 
			\imT(\vec u, \widetilde\Omega
		\setminus \overline C_{\delta_k})\\
			\delta_k(\cos\theta, \sin\theta)& 
			\text{otherwise,}
		\end{cases}
	\end{align}
	is in $SBV(\widetilde \Omega_{\vec b}, \R^2)$,
	with derivative given by
	\begin{multline*}
		D(u_1^{-1}, u_2^{-1})
		\res \imT(\vec u, \widetilde \Omega \setminus \overline C_{\delta_k})
		+ \delta_k (-\sin \theta, \cos\theta)
		\otimes D\theta \res 
		\big (\widetilde\Omega_{\vec b} \setminus 
		\overline{\imT(\vec u, \widetilde \Omega \setminus \overline C_{\delta_k})}\big )
		\\
		+ 
		\big ( (u_1^{-1}, u_2^{-1})^+ - 
		\delta_k(\cos \theta, \sin \theta)\big )
		\vecg\nu_{\imT(\vec u, \widetilde \Omega \setminus \overline C_{\delta_k})}
		\mathcal H^2\res \big (\widetilde\Omega_{\vec b}\cap \partial^*
			\imT(\vec u, \widetilde \Omega \setminus \overline C_{\delta_k}) \big ).
	\end{multline*}
	However,
	by Lemma \ref{le:imTuimTv},
	the 
	radial component 
	of what would be the 
	planar map 
	corresponding to $\vec W_k$
	in \eqref{eq:w_k_3D}
	is precisely $w_k$, so
	by Part IV we know that
	the jump 
	$
	(u_1^{-1}, u_2^{-1})^+ - 
		\delta_k(\cos \theta, \sin \theta)
		$
		is zero
		for $\mathcal H^2$-a.e.\@
		point on 
		$\widetilde\Omega_{\vec b}\cap \partial^*
			\imT(\vec u, \widetilde \Omega \setminus \overline C_{\delta_k})$.
	Therefore,
	the maps $\vec W_k$
	under consideration
	belong to $W^{1,1}(\widetilde\Omega_{\vec b})$.
	
	The uniform convergence $w_k\to w$
	of Part IV translates, in particular,
	 into
	the a.e.\ convergence of the maps $\vec W_k$ in \eqref{eq:w_k_3D}
	to the map $(\widehat{u_1^{-1}},\widehat{u_2^{-1}})$
	in the statement of the proposition. 
	On the other hand, 
	Lemma \ref{le:uvSobolev}
	shows that
	the gradients of the maps $\vec W_k$
	are equiintegrable because  
		\begin{equation*}
	\int_A \left| D \vec W_k \right| \dd \vec y \leq \int_{A\cap \imT(\vec u, \widetilde \O \setminus L)} | D \vec u^{-1} | \, \dd \vec y
	+
	\int_{A\cap  \widetilde \O_{\vec b}} \delta_k |D\theta| \,\dd\vec y
	\end{equation*}
	for any measurable subset $A\subset\pi(\widetilde \Om_{\vec b})$. 
	Therefore, the limit 
	$(\widehat{u_1^{-1}},\widehat{u_2^{-1}})$
	also belongs to $W^{1,1}(\widetilde \Omega_{\vec b}, \R^2)$,
	finishing the proof.
\end{proof}

\section{Weak limits of regular maps}\label{sec:weak_limits}
We investigate here the properties of maps in $\overline{\Asr}$: the weak $H^1$ closure
of the class of regular maps. 
We start by proving that $\overline{\Asr}$ is contained in the space $\mathcal B$ defined in \eqref{eq:defB}.

\begin{theorem}\label{prop:image_ouverte}
Let $\vec u \in \overline{\Asr}$. Then 
\begin{enumerate}[i)]
\item $\vec u$ belongs to $\As$. 

\item\label{item:ouverteii} 

$\imG(\vec u,\Om)= \Om_{\vec b}$ a.e.\@ 
and $\mathcal{L}^3 (\imT(\vec u,L)) = 0$.

\item\label{item:ouverteiii} $\vec u^{-1}\in BV(\widetilde\O_{\vec b},\R^3)$ 
and $ \supp D^s \vec u^{-1} \subset \imT (\vec u,L)$.
Moreover, $\| \vec u^{-1} \|_{BV(\widetilde\Om_{\vec b},\R^3)} \leq M$ for some $M>0$ not depending on $\vec u$.

\item\label{item:ouverteiv} $u^{-1}_\alpha\in W^{1,1}(\widetilde\Om_{\vec b})$ for $\alpha=1,2$.
\end{enumerate}
\end{theorem}

\begin{proof}

	Let $\{ \vec u_n \}_{n\in \N} \subset \Asr$ satisfy $\vec u_n \rightharpoonup \vec u$ in $H^1(\widetilde{\Om},\R^3)$.
	 By Proposition \ref{prop:closedeness_of_Asym}, $\vec u \in \As$,
	 $\det D\vec u_n\weakc \det D\vec u$
	 in $L^1(\tilde \Omega)$,
	 and
	 $\chi_{\imG(\vec u_n, \widetilde \Omega)}\to \chi_{\imG(\vec u, \widetilde \Omega)}$ a.e.
	Now, by Proposition \ref{INV}.(b),
	$\imG(\vec u_n, \widetilde \Omega) = \widetilde\Omega_{\vec b}$ a.e.\@ for 
	every $n\in \N$, 
	hence $\vec u$ inherits this property.
	By \eqref{le:imL_union_imOminusL} and Lemma \ref{le:contained_in_target_domain}\,\ref{eq geo=top}), 
	it then follows that $\mathcal L^3(\imT(\vec u, L))=0$, completing the proof of \emph{\ref{item:ouverteii})}.

From Proposition \ref{INV2} we have that $\vec u_n^{-1} \in W^{1,1}(\widetilde \Om_{\vec b},\R^3)$ for all $n \in \N$ and
\begin{align*}
\|D \vec u_n^{-1}\|_{L^1(\widetilde\Om_{\vec b},\R^{3\times 3})} 
&= \left\| \cof D \vec u_n \right\|_{L^1(\widetilde\Om,\R^{3\times 3})}
\\ &\leq \|D \vec u_n\|^2_{L^2(\Om,\R^{3\times 3})}
+ \|D\vec b\|^2_{L^2(\Omega_D;\R^{3\times 3})}
  \leq E (\vec u_n) + C \leq E (\vec b) + C ,
\end{align*}
where we have used Lemma \ref{eq:ineq-energy-area}.

On the other hand the image of each $\vec u_n^{-1}$ is contained in $\O$, so $\| \vec u_n^{-1} \|_{L^{\infty} (\widetilde{\O}_{\vec b}, \R^3)}$ and, hence $\| \vec u_n^{-1} \|_{L^1 (\widetilde{\O}_{\vec b}, \R^3)}$ are bounded by a constant only depending on $\O$ and $\widetilde{\O}_{\vec b}$.
Thus, by the theorem of compactness in $BV$ we find that, up to a subsequence, there exists $\vec w \in BV(\widetilde\Om_{\vec b},\R^3)$ such that $\vec u_n^{-1} \rightarrow \vec w$ in $L^1(\widetilde \Om_{\vec b},\R^3)$ and a.e.\ in $\widetilde\Omega_{\vec b}$. 
By Lemma \ref{le:convu-1}, $\vec w=\vec u^{-1}$ a.e\@.
Finally, by Lemma \ref{le:u-1} we have $ \supp D^s \vec u^{-1} \subset \imT (\vec u,L)$.
This shows \emph{\ref{item:ouverteiii})}.

Since $\mathcal{L}^3 (\imT(\vec u,L)) = 0$, the functions $u^{-1}_\alpha$ and $\widehat{u^{-1}_\alpha}$ (see Proposition \ref{regularity first 2 components}) coincide a.e\@.
Thus, by Proposition \ref{regularity first 2 components}, $u^{-1}_\alpha\in W^{1,1}(\widetilde \Om_{\vec b})$, which shows \emph{\ref{item:ouverteiv})}.
\end{proof}

For $\vec u\in \overline{\Asr}$ we have, by Theorem \ref{prop:image_ouverte}, that $\vec u^{-1}\in BV(\widetilde\Om_{\vec b},\R^3)$ and we introduce the following standard decomposition of the distributional derivative of $\vec u^{-1}$:
\begin{equation*}
D \vec u^{-1}= \nabla \vec u^{-1} +D^s \vec u^{-1}=\nabla \vec u^{-1}+D^j\vec u^{-1}+D^c \vec u^{-1} ;
\end{equation*}
see, e.g., \cite[Sect.\ 3.9]{AmFuPa00}.
In this decomposition $\nabla \vec u^{-1}$ denotes the absolutely continuous part of $D \vec u^{-1}$ with respect to the Lebesgue measure, $D^s \vec u^{-1}$ is the singular part which can be furthermore decomposed in a jump part $D^j \vec u^{-1}$ and a Cantor part $D^c \vec u^{-1}$.
Moreover, we denote by $J_{\vec u^{-1}}$ the set of jump points of $\vec u^{-1}$.
We fix a Borel orientation $\vec \nu$ of $J_{\vec u^{-1}}$, and, with respect to this orientation, the lateral traces of $\vec u^{-1}$ are denoted by $(\vec u^{-1})^+$ and $(\vec u^{-1})^-$.
Analogously, the jump is defined as $[\vec u^{-1}] := (\vec u^{-1})^+ - (\vec u^{-1})^-$.

The following lemma, which uses many ideas of \cite[Th.\ 2]{HeMo11}, relates the surface energy \(\E_{\vec u}\) with the singular part of \(D \vec u^{-1}\).
With a small abuse of notation, given $\phi \in C^1_c (\O)$, we define $[ \phi \circ \vec u^{-1}]$ in $J_{\vec u^{-1}}$ as $\phi \circ (\vec u^{-1})^+ - \phi \circ (\vec u^{-1})^-$.
Recall that 
 \(\vec u^{-1}\) initially is defined only 
 on \(\imG(\vec u, \widetilde \O)\) 
 but if 
 \(\imG(\vec u,\widetilde\Om)=\widetilde \Om_{\vec b}\) a.e.\@
 then $\vec u^{-1}$ is defined a.e.\@
 in the open set $\widetilde \Omega_{\vec b}$.  Assume that $\vec u^{-1}$ is the precise representative of itself.

\begin{lemma}\label{lem:linkBV}
Let \(\vec u\in H^1 (\widetilde \Om,\R^3) \cap L^\infty(\Om,\R^3) \) be such that \(\det D\vec u\in L^1(\widetilde \Om)\) and $\det D \vec u > 0$ a.e\@. Let \(\phi\in C^1_c(\widetilde\Om) \) and \(\vec g\in C^1_c(\R^3, \R^3) \). 
Suppose that \(\imG(\vec u,\widetilde\Om)=\widetilde \Om_{\vec b}\) a.e., $\vec u$ is injective a.e.\ and
 \(\vec u^{-1} \in BV(\widetilde \Om_{\vec b},\R^3)\).
Then
\begin{align*}
\overline{\E}_{\vec u }(\phi,\vec g) &= - \langle D^s (\phi \circ \vec u^{-1}),\vec g \rangle  \\
& = - \int_{\widetilde \Om_{\vec b}} \nabla \phi(\vec u^{-1}(\vec y))\otimes \vec g(\vec y)\cdot \dd D^c \vec u^{-1}(\vec y)- \int_{J_{\vec u^{-1}}} [ \phi \circ \vec u^{-1}] \, \vec g \cdot \vec \nu \, \dd \mathcal{H}^2 .
\end{align*}
\end{lemma}

\begin{proof}
By the change of variables formula and using that \(\imG(\vec u,\widetilde\Om)=\widetilde \Om_{\vec b}\) a.e., we find
\begin{equation}\label{eq:link1}
 \overline{\E}_{\vec u }(\phi,\vec g) =\int_{\widetilde \Om_{\vec b}} \left[ \vec g (\vec y) \cdot D \vec u (\vec u^{-1} (\vec y))^{-T} D \phi(\vec u^{-1} (\vec y)) + \phi(\vec u^{-1} (\vec y)) \div \vec g(\vec y) \right] \dd \vec y.
\end{equation}
By the chain rule for $BV$ functions (see, e.g., \cite[Th.\ 3.96]{AmFuPa00}), $\phi \circ \vec u^{-1} \in BV (\widetilde{\O}_{\vec b}, \R^3)$ and
\begin{equation}\label{eq:chain}
\begin{split}
 & \nabla (\phi \circ \vec u^{-1}) = \nabla \phi (\vec u^{-1}) \, \nabla \vec u^{-1} , \\
 & D^s (\phi \circ \vec u^{-1}) = \nabla \phi (\vec u^{-1}) D^c \vec u^{-1} + [\phi \circ \vec u^{-1} ] \otimes \vec \nu_{\vec u} \mc{H}^{n-1} \res_{J_{\vec u^{-1}}} .
\end{split}
\end{equation}
By Lemma \ref{le:u-1}, $\nabla \vec u^{-1} (\vec u (\vec x)) = \nabla \vec u (\vec x)^{-1}$ for a.e.\ $\vec x \in \widetilde\O$. This and \eqref{eq:chain} imply that, for a.e.\ $\vec y \in \tilde{\O}_{\vec b}$,
\[
 \nabla (\phi \circ \vec u^{-1}) (\vec y) = \nabla \phi (\vec u^{-1} (\vec y)) \, \nabla \vec u (\vec u^{-1} (\vec y))^{-1} .
\]
Therefore,
\begin{equation}\label{eq:link2}
 \int_{\widetilde\Om_{\vec b}} \vec g (\vec y) \cdot D \vec u (\vec u^{-1} (\vec y))^{-T} D \phi(\vec u^{-1} (\vec y)) \, \dd \vec y 
 = \int_{\widetilde\Om_{\vec b}} \nabla (\phi \circ \vec u^{-1}) (\vec y) \cdot \vec g (\vec y) \, \dd \vec y = \langle \nabla (\phi \circ \vec u^{-1} ), \vec g \rangle .
\end{equation}
On the other hand, by definition of distributional derivative,
\begin{equation}\label{eq:link3}
 \int_{\widetilde\Om_{\vec b}} \phi(\vec u^{-1} (\vec y)) \div \vec g(\vec y) \, \dd \vec y = - \langle D (\phi \circ \vec u^{-1}),\vec g \rangle .
\end{equation}
Putting together \eqref{eq:link1}, \eqref{eq:link2} and \eqref{eq:link3} we obtain
\[
 \overline{\E}_{\vec u }(\phi,\vec g) = \langle \nabla (\phi \circ \vec u^{-1} ) - D (\phi \circ \vec u^{-1}),\vec g \rangle = - \langle D^s (\phi \circ \vec u^{-1} ) ,\vec g \rangle .
\]
This, together with \eqref{eq:chain}, concludes the proof.
\end{proof}

Recall the definition of the singular segment 
$L=\overline \Omega \cap \R\vec e_3$ in
\eqref{def:segment_L}.

\begin{proposition}
Let \( \vec u\in \overline{\Asr}\).
Then $\vec u^{-1}(\vec y) \in L$ for $|D^c \vec u^{-1}|$-a.e.\ $\vec y \in \widetilde\Om_{\vec b}$ and 
\break
$(\vec u^{-1})^{\pm} (\vec y) \in L$ for $\mathcal{H}^2$-a.e.\ $\vec y \in J_{\vec u^{-1}}$.
\end{proposition}
\begin{proof}
Without loss of generality, $\vec u^{-1}$ is the precise representative of itself.

Let $\vec g \in C^1_c(\R^3,\R^3)$ and $\phi \in C^1_c(\widetilde\O \setminus \R\vec e_3)$.
Then, there exists $\d>0$ such that $\phi \in C^1_c(\widetilde\O \setminus \overline{C}_{\d})$.
By Lemma \ref{N condition}, we have that $\overline{\E}_{\vec u} (\phi,\vec g)=0$, so due to Lemma \ref{lem:linkBV} we obtain that 
\begin{equation}\label{eq:D^s0}
 \int_{\widetilde\Omega_{\vec b}} \nabla \phi(\vec u^{-1}(\vec y))\otimes \vec g(\vec y) \cdot \dd D^c \vec u^{-1}(\vec y) + \int_{J_{\vec u^{-1}}} [ \phi \circ \vec u^{-1}] \, \vec g \cdot \vec \nu \, \dd \mathcal{H}^2=0.
\end{equation}
By approximation, the previous equality, which does not involve derivatives of $\vec g$, is also valid for every bounded Borel $\vec g : \R^3 \to \R^3$.

Let $D^c \vec u^{-1}= \vec A \, |D^c \vec u^{-1}|$ be the polar decomposition of $D^c \vec u^{-1}$, so $\vec A : \widetilde{\O}_{\vec b} \to \R^{3 \times 3}$ is Borel, $|D^c \vec u^{-1}|$-integrable and $|\vec A| = 1$ in $|D^c \vec u^{-1}|$-a.e.\ $\widetilde{\O}_{\vec b}$.
Let $\vec y_0 \in \widetilde{\O}_{\vec b}$ be a $|D^c \vec u^{-1}|$-Lebesgue point of $\vec u^{-1}$, i.e.,
\begin{equation}\label{eq:DcLebesgue}
 \lim_{r \to 0^+} \frac{\int_{B(\vec y_0,r)} |\vec u^{-1} (\vec y)- \vec u^{-1} (\vec y_0) | \, \dd |D^c\vec u^{-1}|(\vec y)}{|D^c\vec u^{-1}|(B(\vec y_0,r))} = 0 ,
\end{equation}
and note that $|D^c \vec u^{-1}|$-a.e.\ point of $\widetilde{\O}_{\vec b}$ satisfies that.
Let us suppose that $\vec u^{-1}(\vec y_0) \notin L$ and take a closed cube $Q \subset \widetilde \O$ centered at $\vec u^{-1}(\vec y_0)$ with $Q \cap L = \varnothing$.
Consider the Borel set
\[
 U:=\{\vec y \in \widetilde{\O}_{\vec b} \colon \vec u^{-1}(\vec y)\in Q \text{ and } \vec u^{-1} \text{ is approximately continuous at } \vec y \} .
\]

Given any $\psi \in C^1 (\R)$, take $\phi \in C^1_c (\widetilde \O \setminus \R\vec e_3)$ such that $\phi(\vec x)=\psi(x_\alpha)$ for all $\vec x \in Q$. For $1\leq \alpha,i \leq 3$ and $r>0$ fixed, we apply \eqref{eq:D^s0} to $\vec g = \sgn \psi' \sgn A_{\alpha i} \chi_{B(\vec y_0,r) \cap U} \vec e_i$ and deduce that 
\begin{equation*}
 \int_{\{ \vec y\in B(\vec y_0,r): \vec u^{-1}(\vec y)\in Q\}} \sgn \psi' \sgn A_{\alpha i} \left( \nabla \phi (\vec u^{-1}(\vec y)) \otimes \vec e_i \right) \cdot \vec A \, \dd|D^c\vec u^{-1}|=0.
\end{equation*}
This can also be written as
\begin{equation*}
\int_{\{ \vec y\in B(\vec y_0,r): \vec u^{-1}(\vec y)\in Q\}} \left|\psi' ( u^{-1}_\alpha(\vec y)) \right| \left| A_{\alpha i} \right| \dd | D^c \vec u^{-1} | =0 .
\end{equation*}
We use the previous equality first with $\psi(t)=\cos t $ and then with $\psi(t) =\sin t$.
We sum the two equalities and use that \(|\cos t|+|\sin t|\geq 1 \) to get 
\begin{equation*}
 \int_{\{ \vec y\in B(\vec y_0,r): \vec u^{-1}(\vec y)\in Q\}} |A_{\alpha i}| \, \dd |D^c \vec u^{-1}|=0.
\end{equation*}
We then sum this equality for $1\leq \alpha,i\leq 3$ to obtain
\begin{equation*}
|D^c \vec u^{-1}|\left(\{\vec y\in B(\vec y_0,r): \vec u^{-1}(\vec y) \in Q \} \right)=0.
\end{equation*}
This equality implies that $|\vec u^{-1}(\vec y)-\vec u^{-1}(\vec y_0)| > \text{diam } Q/2$ for $|D^c\vec u^{-1}|$-a.e.\ $\vec y\in B(\vec y_0,r)$.
Being true for all $r>0$, this is a contradiction with \eqref{eq:DcLebesgue}.
Therefore, $|D^c \vec u^{-1}|$-a.e.\ $\vec y \in \widetilde{\O}_{\vec b}$ satisfies $\vec u^{-1} (\vec y) \in L$.

Now we show that $(\vec u^{-1})^{\pm} (\vec y) \in L$ for $\mathcal{H}^2$-a.e.\ $\vec y \in J_{\vec u^{-1}}$.
Let $\vec y_0 \in J_{\vec u^{-1}}$ be a $\mc{H}^2 \res_{J_{\vec u^{-1}}}$-Lebesgue point for both $(\vec u^{-1})^+$ and $(\vec u^{-1})^-$, i.e.,
\begin{equation}\label{eq:H2Lebesgue}
 \lim_{r \to 0^+} \frac{\int_{J_{\vec u^{-1}} \cap B(\vec y_0,r)} |(\vec u^{-1})^{\pm} (\vec y) - (\vec u^{-1})^{\pm} (\vec y_0) | \, \dd \mc{H}^2 (\vec y)}{\mc{H}^2 (J_{\vec u^{-1}} \cap B(\vec y_0,r))} = 0 ,
\end{equation}
and note that $\mc{H}^2$-a.e.\ point in $J_{\vec u^{-1}}$ satisfies that.
For each $r>0$, we apply \eqref{eq:D^s0} to $\vec g = \chi_{J_{\vec u^{-1}} \cap B(\vec y_0,r)} \vec \nu$ and deduce that 
\begin{equation*}
 \int_{J_{\vec u^{-1}} \cap B(\vec y_0,r)} [ \phi \circ \vec u^{-1}] \, \dd \mathcal{H}^2 = 0 .
\end{equation*}
This and \eqref{eq:H2Lebesgue} imply that $[ \phi \circ \vec u^{-1}] (\vec y_0) = 0$.
If $(\vec u^{-1})^+ (\vec y_0) \notin L$ and $(\vec u^{-1})^- (\vec y_0) \in L$, we choose $\phi \in C^1_c (\widetilde\O \setminus \R\vec e_3)$ such that $\phi (\vec u^{-1})^+ (\vec y_0) \neq 0$ and reach a contradiction with $[ \phi \circ \vec u^{-1}] (\vec y_0) = 0$.
Analogously if $(\vec u^{-1})^+ (\vec y_0) \in L$ and $(\vec u^{-1})^- (\vec y_0) \notin L$.
If $(\vec u^{-1})^+ (\vec y_0) \notin L$ and $(\vec u^{-1})^- (\vec y_0) \notin L$, we choose $\phi \in C^1_c (\widetilde\O \setminus \R\vec e_3)$ such that $\phi ((\vec u^{-1})^+ (\vec y_0)) \neq \phi ((\vec u^{-1})^- (\vec y_0))$, which contradicts $[ \phi \circ \vec u^{-1}] (\vec y_0) = 0$.
Hence, the only possibility is that $(\vec u^{-1})^+ (\vec y_0) \in L$ and $(\vec u^{-1})^- (\vec y_0) \in L$.
\end{proof}

\section{Lower bound for the relaxed energy and an explicit alternative variational problem}\label{sec:lower_bound}

In this section we study the energetic cost for a weak limit of functions in \(\Asr\) to leave \(\Asr\). 
Note that, by Conti-De Lellis' counterexample, condition INV is not satisfied, in general, by functions in \(\overline{\Asr}\). 
This is due to the lack of equiintegrability of the cofactors, so the theory of \cite{HeMo12} cannot be applied.

A standard diagonal argument shows that $\overline{\Asr}$ is closed under the weak convergence of $H^1 (\O, \R^3)$.
From Theorem \ref{prop:image_ouverte} we see that the energy 
\begin{align}
\label{def:energy_F}
 F(\vec u):=E(\vec u) +2 |D^s u^{-1}_3|(\widetilde{\O}_{\vec b}).
 \end{align}
 is well defined on \( \overline{\Asr} \).
 We start with the following lemma, which plays a role of an energy-area inequality and should be compared with Lemma \ref{eq:ineq-energy-area}. 

\begin{lemma}\label{lem:inequality_energy_area_true}
 Let \( \vec u\in \As\). Then $\left| \adj D\vec u \, \vec e_3 \right| \leq \frac12 |D \vec u|^2$.
This inequality is optimal and cannot be attained by a map in \(\As\).
\end{lemma}

\begin{proof}
 With the expressions of \(D \vec u\) and \(\cof D \vec u\) in terms of the associated \(2D\) map \( \vec v\),  cf.\ \eqref{eq:comatrix_cylindrical_cylindrical},  we find
 \begin{equation*}
 \left| \adj D\vec u \, \vec e_3 \right| = \frac{|v_1|}{r}\left( |\p_rv_1|^2+|\p_{x_3}v_1|^2\right)^{1/2} \leq \frac12 \left( \frac{|v_1|^2}{r^2}+|\p_rv_1|^2+|\p_{x_3}v_1|^2 \right) \leq\frac12 |D \vec u|^2.
\end{equation*}
The equality implies \( \frac{v_1}{r}=(|\p_r v_1|^2+|\p_{x_3}v_1|^2)^{1/2}\) and \(\nabla v_2=0\).  This cannot be attained by a map in \(\As\), since $\nabla v_2 =0$ implies  $\det D \vec v = 0$, so $\det D \vec u = 0$.
Nonetheless, the constant is optimal in $\As$, as can be checked by considering $v_1 (r, x_3) = r$ and $v_2 (r, x_3) = \e x_3$ for $\e\searrow 0$, which corresponds to $\vec u (\vec x) = (x_1, x_2, \e x_3)$.
\end{proof}

The following lower semicontinuity result is the cornerstone of the strategy in this paper for the study of the regularity of the minimizers of the neo-Hookean energy.

\begin{proposition}\label{prop:F_lower_semi_cont}
The energy \(F\) defined in \eqref{def:energy_F} is  sequentially lower semicontinuous in \(\overline{\Asr}\) for the weak convergence in \(H^1 (\widetilde{\O}, \R^3)\).
\end{proposition}

\begin{proof}
Recall from Theorem \ref{prop:image_ouverte} that \(\overline{\Asr} \subset \As\).
Let \( \{ \vec u_k \}_{k \in \N} \) be a sequence in \(\overline{\Asr}\) tending weakly in $H^1 (\widetilde{\O}, \R^3)$ to $\vec u \in \overline{\Asr}$.
Thanks to Theorem \ref{prop:image_ouverte}\,\ref{item:ouverteiii}), the $BV$ norm of \( \vec u_k^{-1} \) is bounded, so, due to Lemma \ref{le:convu-1}, we have that, up to a subsequence, \(\vec u_k^{-1} \weakcs \vec u^{-1} \) in \(BV(\widetilde{\Om}_{\vec b}, \R^3)\) and a.e\@.
By Proposition \ref{prop:closedeness_of_Asym}, we have that \(\det D\vec u_k \rightharpoonup \det D 
\vec u\) in \(L^1(\widetilde{\Om})\) and, because of the convexity of $H$, 
\begin{equation}\label{eq:Hdetuk}
\int_{\widetilde{\Om}} H(\det D \vec u)\leq \liminf_{k\rightarrow \infty }\int_{\widetilde{\Om}} H(\det D\vec u_k). 
\end{equation}

We first prove that the sequence $\{ \det D\vec u_k^{-1} \}_{k \in \N}$ is equiintegrable. This can be proved as in \cite[Prop.\ 7.8]{BaHeMo17}.
Indeed,
define $H_1:(0,\infty)\to \R$ as $H_1(t):=tH(1/t)$. Then $H_1$ grows superlinearly at infinity and 
$$\int_{\widetilde{\Omega}_{\vec b}} H_1(\det D\vec u_k^{-1}) \, \dd\vec y = \int_{\widetilde{\Omega}} H(\det D\vec u_k) \, \dd \vec x \leq E(\vec b).$$
Thus the equiintegrability follows from the De La Vall\'ee Poussin criterion.

Now, let $\eps >0$. Recall that \(C_\delta\) is given by \eqref{def:Cdelta}.  Choose $\delta_0>0$ such that 
\begin{align}
    \label{eq:pLsc-1}
  \int_{C_{\delta_0}\cap \widetilde{\O}} |D\vec u|^2 \, \dd\vec x <\eps.
\end{align}

Because of the axial symmetry, it can be seen that the sequence $\{ \chi_{\widetilde{\Omega}\setminus C_{\delta_0}} \cof D\vec u_k \}_{k \in \N}$
is equiintegrable, cf.\ \cite[Th.\ 1.3]{HeRo18}. This is due to the fact that the corresponding 2D maps \( \vec v_k\) are bounded in \(H^1 \big(\pi (\widetilde{\Omega} \setminus \overline C_{\delta_0} ),\R^2\big)\) and then by a result of M\"uller \cite{Muller90}, since we also have \(\det D \vec v_k>0\) a.e., \( \det D \vec v_k\) are equiintegrable. Now we obtain the equiintegrability result for $\{ \chi_{\widetilde{\Omega}\setminus C_{\delta_0}} \cof D\vec u_k \}_{k \in \N}$ by expressing the cofactor matrix in terms of the 2D map \( \vec v_k\) and observing that one entry is \( \det D \vec v_k\) and the others are products of a sequence converging strongly in \(L^2\) by a sequence converging weakly in \(L^2\); cf.\ \eqref{eq:comatrix_cylindrical_cylindrical} in the Appendix.

Hence, there exists $\eta>0$, independent of $k$,  such that if \(A \subset \widetilde{\O}\) is measurable,
\begin{align} 
  \label{eq:pLsc-2}
  |A|<\eta\ \Rightarrow\ \int_{A\setminus C_{\delta_0}} \left| \cof D\vec u_k \right| \dd\vec x < \eps, \quad \forall k \in \mathbb{N}.
\end{align}

Given any open subset $V$ of $\widetilde{\Omega}_{\vec b}$ (which we shall later choose to be a thin neighbourhood of $\imT(\vec u, L)$), and any good $\delta_1<\delta_0$,
\begin{align}
    \label{eq:pLsc-5}
  \int_V |\nabla (\vec u_k^{-1})_3| \, \dd\vec y = 
   \int_{\vec u_k^{-1}(V)\cap C_{\delta_1}}|\adj \nabla \vec u_k \, \vec e_3|  \, \dd\vec x + \int_{\vec u_k^{-1}(V)\setminus C_{\delta_1}} |\adj \nabla \vec u_k \, \vec e_3| \, \dd\vec x.
\end{align}
By Lemma \ref{lem:inequality_energy_area_true} the first integral 
in the right-hand side of \eqref{eq:pLsc-5} is bounded by the integral of $\frac12 |D\vec u_k|^2$ in $C_{\delta_1}\cap \widetilde{\O}$.
As for the second integral, note that
\begin{align}
    \label{eq:pLsc-3}
   |\vec u_k^{-1}(V) \setminus C_{\delta_1}| \leq \int_{V} \det D\vec u_k^{-1} \, \dd\vec y.
\end{align}

	Since 
	$\mathcal L^3 \big (\imT(\vec u, L)\big )=0$
	combining \eqref{eq:pLsc-3}, \eqref{eq:pLsc-2}, and the equiintegrability of 
	$\{ \det D\vec u_k^{-1} \}_{k \in \N}$
	it is possible to find $\delta_1>0$, with \(\delta_1<\delta_0\) and an open set $V\subset \widetilde{\Omega}_{\vec b}$ such that 
	\begin{align}
	  \label{eq:pLsc-4}
	  \imT(\vec u, L)\subset V
	  \quad\text{and}\quad
	  \int_{\vec u_k^{-1}(V)\setminus C_{\delta_1}} \left| \cof D\vec u_k \right| \dd\vec x < \eps,  \quad \forall k \in \mathbb{N}.
	\end{align}

By \eqref{eq:pLsc-4}, for this $V$ we get
\begin{equation*}
  \int_V |\nabla (\vec u_k^{-1})_3| \, \dd\vec y
  \leq \frac12 \int_{ C_{\delta_1}\cap \widetilde{\O}} |D\vec u_k|^2 \, \dd\vec x + \eps,
\end{equation*}
and therefore
\begin{align*}
 |D (\vec u_k^{-1})_3|(V) & 
 = \int_{V} |\nabla (\vec u_k^{-1})_3| \, \dd \vec y +|D^s (\vec u_k^{-1})_3|(V) \\
 &  \leq \frac12 \int_{C_{\delta_1}\cap \widetilde{\O}} |D\vec u_k|^2 \, \dd\vec x + \eps 
 +|D^s (\vec u_k^{-1})_3|(\widetilde{\O}_{\vec b}).
\end{align*}
Observe that by Theorem \ref{prop:image_ouverte}\,\ref{item:ouverteiii}) the inclusions $ \supp D^s \vec u^{-1} \subset \imT (\vec u,L)\subset V$ hold.
Then, by the inequality above, as $\vec u_k^{-1} \weakcs \vec u^{-1}$ in $BV(\widetilde{\O}_{\vec b},\R^3)$ we have that 
\begin{equation}\label{eq:relax2'}
| D^s (\vec u^{-1})_3|(\widetilde\Omega_{\vec b}) \leq \eps + \liminf_{k \to \infty} \left[ \frac12 \int_{C_{\d_1}\cap \widetilde{\O}} |D \vec u_k|^2 \, \dd \vec x 
+|D^s (\vec u_k^{-1})_3|(\widetilde{\O}_{\vec b}) \right] .
\end{equation}
On the other hand, by \eqref{eq:pLsc-1}, as $\vec u_k \weakc \vec u$ in $H^1 (\widetilde{\O}, \R^{3 \times 3})$  
we have also that
\begin{align}\label{eq:relax4'}
	   \frac{1}{2}
   \int_{C_{\delta_1}\cap \widetilde{\O}} |D\vec u|^2 \, \dd \vec x +
	   \frac{1}{2}
   \int_{\widetilde{\O} \setminus C_{\d_1}} |D\vec u|^2 \, \dd \vec x 
 \leq \eps +  \liminf_{k\rightarrow \infty} \frac12 \int_{\widetilde{\O} \setminus C_{\d_1}} |D \vec u_k|^2 \, \dd \vec x. 
\end{align}
Gathering \eqref{eq:relax2'} and \eqref{eq:relax4'}, since $\e>0$ is arbitrary 
we obtain  that
\begin{equation}\label{eq:Important}
 \frac12 \int_{\widetilde{\O}} |D\vec u|^2 +|D^s(\vec u^{-1})_3|(\widetilde{\O}_{\vec b})
 \leq \liminf_{k\rightarrow \infty} \left[ \frac12 \int_{\widetilde{\O}} |D \vec u_k|^2  
 +|D^s (\vec u_k^{-1})_3|(\widetilde{\O}_{\vec b}) \right].
\end{equation} 
The proof of the proposition is concluded by gathering \eqref{eq:Hdetuk} and \eqref{eq:Important}.
\end{proof}

\begin{remark}
Without the Sobolev regularity for the horizontal components of the inverse, the estimate of the 
first term of the right-hand side of \eqref{eq:pLsc-5} would have been made for the whole cofactor matrix, yielding in
\eqref{def:energy_F} the suboptimal prefactor $\sqrt{3}$ of Lemma \ref{eq:ineq-energy-area}
instead of the prefactor $2$ coming from Lemma \ref{lem:inequality_energy_area_true}.
\end{remark}

\begin{remark}
From Proposition \ref{prop:F_lower_semi_cont} we get  in particular that
if  \(\vec u_k\) is a sequence in \(\Asr\) with $H^1$-weak limit $\vec u$, then 
$$F(\vec u) \leq \liminf_{k\rightarrow \infty} E(\vec u_k).$$ 
\end{remark}

\smallskip
Since we are in the presence of a problem of lack of compactness it is natural to try and describe the space $\oAsr$ and the relaxed energy defined on this space by
\begin{equation*}
 E_{\rel}(\vec u):= \inf \{ \liminf_{k\rightarrow \infty} E(\vec u_k) : \{ \vec u_k \}_{k \in \N} \in \Asr \text{ and } \vec u_k \rightharpoonup \vec u \text{ in } H^1 (\O, \R^3) \}.
\end{equation*}

It is well known that $E_{\rel}$ is the largest lower semicontinuous functional in $\overline{\Asr}$ (for the $H^1$-weak topology) that is below $E$ in $\Asr$.
Since $F$ is lower semicontinuous in $\overline{\Asr}$ and
\begin{equation*}
 E_{\rel} = E \quad \text{in } \Asr
\end{equation*}
we conclude that
\begin{equation}\label{eq:lowerbound}
 E_{\rel}\geq F \quad \text{in } \overline{\Asr} .
\end{equation}

It is tempting to conjecture that the equality \(E_{\rel}=F\) holds at least for some special choices of the function \(H\). 
In view of Proposition \ref{prop:F_lower_semi_cont} (and its consequence \eqref{eq:lowerbound}), 
it remains to characterize $\oAsr$ and to show that for any $\vec u \in \oAsr$ there exists a sequence $\{ \vec u_n \}_{n \in \N} \subset \Asr$ 
converging weakly to $\vec u$ in $H^1 (\O, \R^3)$ such that
\[
 \lim_{n \to \infty} E (\vec u_n) = F (\vec u) .
\]
There are serious difficulties in constructing this sequence $\{ \vec u_n \}_{n \in \N}$ (if it exists at all).
One of them relies on the restrictions of being orientation-preserving and injective a.e., even though there are some partial results in this direction (see \cite{IwKoOn11,HenclMo15,CoDo15,MoOl19,DePr19} and the references therein).

At any rate, the interest of defining the relaxed energy in an abstract way is to be able to prove that it attains its infimum in \(\overline{\Asr}\), and that the initial energy attains its minimum in \(\Asr\) if and only if there exists a minimizer of \(E_{\rel}\) in \(\overline{\Asr}\) which is in \(\Asr\).
These two facts are classical in the theory of relaxation and follow from abstract arguments.
The energy \(F\)  satisfies analogous properties and, hence, can be a substitute of \(E_{\rel}\).

\begin{theorem}\label{th:main3}
The energy $F$ has a minimizer  in \(\overline{\Asr}\).
If it belongs to \(\Asr\), then it is also a minimizer of \(E\).
\end{theorem}

\begin{proof}
Recall that \(\overline{\Asr}\) is closed for the weak convergence in \(H^1 (\widetilde{\O}, \R^3)\).
It is also bounded in \(H^1 (\widetilde{\O}, \R^3)\).
From Proposition \ref{prop:F_lower_semi_cont}, \(F\) is lower semicontinuous in \(\overline{\Asr}\).
Clearly, $F$ is  coercive in \(\overline{\Asr}\).
This readily implies the existence of minimizers.

As for the second part of the statement, we assume that there exists a minimizer \(\vec u_0\) of \(F\) in \(\overline{\Asr}\) such that \(\vec u_0\in \Asr\).
We then have \( F(\vec u_0)\leq F(\vec w)\) for any \(\vec w\in \overline{\Asr}\).
But since \(F =E \) in \(\Asr\), we find that \(E(\vec u_0)\leq E(\vec w)\) for all \(\vec w\in \Asr\).
That is, \(\vec u_0\) is a minimizer of \(E\) in \(\Asr\).
\end{proof}

In the same vein, we have the following result.
\begin{proposition}\label{pr:Eattainsinfweakclosure}
The energy $E$ has a minimizer  in \(\overline{\Asr}\).
\end{proposition}
\begin{proof}
From Lemma \ref{le:weakconv}  we have that \(\overline{\Asr}  \subset \As\) and from Proposition  \ref{prop:closedeness_of_Asym} 
that \(E\) is lower semicontinuous on \(\As \).
Moreover, \(\overline{\Asr}\) is closed for the $H^1$-weak convergence. As noted before, $E$ is coercive in  \(\As \).
These are the three main ingredients to obtain the conclusion.
\end{proof}

It would be nice to have an explicit description of \(\overline{\Asr}\).
Although this characterization is missing, we are able to prove the existence of minimizers of the energy \(F\) in 
the explicit space $\mathcal{B}$
defined in
\eqref{eq:defB}
which is a priori larger than \(\overline{\Asr}\).
Indeed, from Theorem \ref{prop:image_ouverte} we have that \(\overline{\Asr}\subset \mathcal{B} \subset \As\).
Besides, the energy \(F\) is well defined on \(\mathcal{B}\),
it controls the $BV$ norm of the inverses, and 
a slight adaptation of Proposition \ref{prop:F_lower_semi_cont}
yields the lower semicontinuity of $F$ in $\mathcal{B}$.

\begin{proposition}\label{prop:F_lower_semi_contB}
The energy \(F\) is  sequentially lower semicontinuous in \(\mc{B}\) for the $H^1$-weak convergence.
\end{proposition}

\begin{proof}
Let \( \{ \vec u_k \}_{k \in \N} \) be a sequence in \(\mc{B}\) tending weakly in $H^1 (\widetilde{\O}, \R^3)$ to $\vec u \in \mc{B}$.
We can assume that $\liminf_{k \to \infty} F(\vec u_k) < \infty$.
In particular, $\sup_{k \in \N} | D \vec u^{-1}_k |(\widetilde{\O}_{\vec b}) < \infty$.
As $\| \vec u_k \|_{L^{\infty} (\widetilde{\O}, \R^3)}$ and, hence, $\| \vec u_k \|_{L^1 (\widetilde{\O}, \R^3)}$ are bounded, the $BV$ norm of \( \vec u_k^{-1} \) is bounded, so, due to Lemma \ref{le:convu-1}, we have that, up to a subsequence, \(\vec u_k^{-1} \weakcs \vec u^{-1} \) in \(BV(\widetilde{\Om}_{\vec b}, \R^3)\) and a.e\@.
From here, the proof is the same as in Proposition \ref{prop:F_lower_semi_cont}.
\end{proof}

\begin{proof}[Proof of Theorem \ref{th:main_theorem_introduction}]
Let \( \{\vec u_k\}_k\) be a minimizing sequence for \(F\) in \(\mathcal{B}\).
Clearly \(\sup_k F(\vec u_k)<\infty\), and we can assume that \(\vec u_k \rightharpoonup \vec u\) in \(H^1 (\widetilde{\O}, \R^3)\).
Since, by Proposition \ref{prop:F_lower_semi_contB}, \(F\) is lower semicontinuous for the weak convergence in \(H^1\) of maps in \(\mc{B}\), it suffices to show that the weak limit \(\vec u\) is in \(\mathcal{B}\).
We know from Proposition \ref{prop:closedeness_of_Asym} that \(\vec u\in \As\), \(\det D\vec u_k \rightharpoonup \det D\vec u\) 
in \(L^1(\widetilde{\Om})\) and \( \imG(\vec u_k,\widetilde{\Om}) \to\imG(\vec u,\widetilde{\Omega})\) a.e\@.
In particular, \( \widetilde{\Omega}_{\vec b} = \imG(\vec u,\widetilde{\Om})\) a.e.\ and from \eqref{le:imL_union_imOminusL}, \(\imT(\vec u,L)\) is a null Lebesgue set. 
Now we use Lemma \ref{eq:ineq-energy-area} to show that
\begin{align*}
F(\vec u_k) &\geq \int_{\widetilde{\O}} \left| \cof D \vec u_k \right| \dd \vec x 
+2|D^s \vec u_k^{-1}|(\widetilde{\O}_{\vec b})
\\ &\geq \int_{\imG(\vec u,\widetilde{\O})}|\nabla \vec u_k^{-1}| \, \dd \vec y
+|D^s \vec u_k^{-1}|(\widetilde{\O}_{\vec b})
=  |D\vec u_k^{-1} |(\widetilde{\O}_{\vec b}).
\end{align*}
As $\{ \vec u_k^{-1} \}_{k \in \N}$ is bounded in $L^{\infty} (\widetilde{\O}_{\vec b}, \R^3)$, 
we find that $\vec u_k^{-1}$ is bounded in $BV(\widetilde{\O}_{\vec b},\R^3)$.
Up to a subsequence, thanks to Lemma \ref{le:convu-1}, we have that $\vec u_k^{-1}  \rightarrow \vec u^{-1}$ in $L^1(\widetilde{\O}_{\vec b},\R^3)$ and a.e., with $\vec u^{-1} \in BV(\widetilde{\O}_{\vec b},\R^3)$. From Proposition \ref{regularity first 2 components} we also infer that \( u_1^{-1}, u_2^{-1}\) are in \(W^{1,1}(\widetilde{\O}_\vec b)\).
This proves that $\vec u$ minimizes $F$ in $\mathcal{B}$.

The other statement of Theorem \ref{th:main_theorem_introduction} can be shown as in the proof of Theorem \ref{th:main3}.
\end{proof}

The following is a summary of existence results we obtained in this article.

\begin{center}
\def\arraystretch{1.3}
\begin{tabular}{|c|c|c|c|c|c|}
\hline
\text{Spaces }& $\Asr$ & $ \overline{\Asr}$ & $\overline{\Asr}$ & $\mathcal{B}$ & $\As$ \\
\hline
\text{Energies} &  $E$ & $E$ & $F$ & $F$ & $E$ \\
\hline
\text{Minimizers} & \text{?} & \text{yes} Prop.\ \ref{pr:Eattainsinfweakclosure} &  \text{yes} Th.\ \ref{th:main3}    &  \text{yes} Th.\ \ref{th:main_theorem_introduction} & \text{yes} Th.\ \ref{th:main1}\\
\hline
\end{tabular}
\end{center}

\section*{Appendix: Working with axially symmetric maps}

We recall from the Appendix in \cite{HeRo18} that if \(\vec u: \O \rightarrow \R^3\) is axisymmetric and is given in cylindrical coordinates 
by \( \vec u(r\cos \theta, r\sin \theta, x_3)= v_1(r,x_3) \vec e_r +v_2(r,x_3) \vec e_3\) then
\begin{equation}\label{eq:comatrix_cylindrical_cylindrical}
\begin{split}
 & D\vec u= \begin{pmatrix}
\p_rv_1 & 0 & \p_{x_3} v_1 \\
0 & \frac{v_1}{r} & 0 \\
\p_rv_2 & 0 & \p_{x_3}v_2
\end{pmatrix}, \qquad
\cof D \vec u = \begin{pmatrix}
\frac{v_1}{r}\p_{x_3}v_2 & 0 & -\frac{v_1}{r}\p_rv_2 \\
0 & \det D v & 0 \\
-\frac{v_1}{r}\p_{x_3}v_1 & 0 & \frac{v_1}{r}\p_rv_1
\end{pmatrix}, \\
 & \det D\vec u = \frac{1}{r} v_1 \det D \vec v,
\end{split}
\end{equation}
and the Dirichlet energy is given by
\begin{equation*}
 \int_{\Om} |D \vec u|^2 \, \dd \vec x = 2\pi \int_{\pi(\O)}\left(|\p_r \vec v|^2+|\p_{x_3} \vec v|^2\right)r \, \dd r \, \dd x_3 +2\pi \int_{\pi(\O)}\frac{v_1^2}{r} \, \dd r \, \dd x_3.
\end{equation*}

\section*{Acknowledgements}
We gratefully acknowledge J.~Ball for his observation (see Section \ref{sec:existence_as})
that in the axisymmetric setting cavitation does not truly show that the neo-Hookean energy
fails to be $H^1$-quasiconvex.

Marco Barchiesi has been supported by project VATEXMATE.
Duvan Henao has been funded by the FONDECYT project 1190018.
Carlos Mora-Corral has been supported by the Agencia Estatal de Investigaci\'on of the Spanish Ministry of Research and Innovation, 
through project PID2021-124195NB-C32 and the Severo Ochoa Programme for Centres of Excellence in R\&D CEX2019-000904-S, 
by the Madrid Government (Comunidad de Madrid, Spain) under the multiannual Agreement with UAM in the line for the Excellence 
of the University Research Staff in the context of the V PRICIT (Regional Programme of Research and Technological Innovation), and by the ERC Advanced Grant 834728.
R\'emy Rodiac has been partially supported by the ANR project BLADE Jr. ANR-18-CE40-0023.\\ 

There is no conflict of interest and there is no data attached to this manuscript.
{\small
\bibliography{biblio} \bibliographystyle{siam}
}

\end{document}